\documentclass[11pt]{article}

\usepackage[]{float,latexsym,times}
\usepackage{amsfonts,amstext,amsmath,amssymb,amsthm}
\usepackage{caption2}

\usepackage{graphicx}

\long\def\symbolfootnote[#1]#2{\begingroup%
\def\thefootnote{\fnsymbol{footnote}}\footnote[#1]{#2}\endgroup}
\usepackage{titlesec} % Very fancy section title control: Sets sections as SQA requires

\titleformat{\section}{\large\bfseries}{\thesection.}{.5em}{}
\titlespacing*{\section}{0pt}{*3}{*2}
\titleformat{\subsection}{\normalfont\bfseries}{\thesubsection.}{.5em}{}
\titlespacing*{\subsection} {0pt}{*3}{*2}
\titleformat{\subsubsection}{\normalfont\bfseries}{\thesubsubsection.}{.5em}{}
\titlespacing*{\subsubsection} {0pt}{*3}{*2}

\theoremstyle{plain} %% italic text
\newtheorem{theorem}{Theorem}[section]
\newtheorem{lemma}{Lemma}[section]
\newtheorem{corollary}{Corollary}[section]

\theoremstyle{definition} %% or \theoremstyle{remark} will produce roman text

\newcommand{\dint}{\mathrm{d}}

\numberwithin{equation}{section} %% double numbering within sections

\usepackage[authoryear]{natbib}
\begin{document}

\title{\textbf{\Large A Linear Programming Approach to Sequential Hypothesis Testing}}

\date{}

\maketitle

%%%%%%%%% Authors, affiliations %%%%%%%%%%%%%%%%%%%%%%%%%%

\author{
	\begin{center}
		\vskip -1cm
		\textbf{\large Michael Fau\ss{} and Abdelhak M. Zoubir} \\
		Signal Processing Group, Technische Universit\"at Darmstadt, \\
		Darmstadt, Germany
	\end{center}
}

\symbolfootnote[0]{\normalsize Address correspondence to Michael Fau\ss{},
	Signal Processing Group, Technische Universit\"at Darmstadt, Merckstra\ss{}e 25, 
	64283 Darmstadt, Germany; E-mail: michael.fauss@spg.tu-darmstadt.de
}

{\small \noindent\textbf{Abstract:}
	Under some mild Markov assumptions it is shown that the problem of designing optimal sequential tests for two simple hypotheses can be formulated
  as a linear program. This result is derived by investigating the La\-gran\-gian dual of the sequential testing problem, which is an
  unconstrained optimal stopping problem depending on two unknown Lagrangian multipliers. It is shown that the derivative of the optimal cost
  function, with respect to these multipliers, coincides with the error probabilities of the corresponding sequential test. This  property is used
  to formulate an optimization problem that is jointly linear in the cost function and the Lagrangian multipliers and can be solved for both with
  off-the-shelf algorithms. To illustrate the procedure, optimal sequential tests for Gaussian random sequences with different dependency structures
  are derived, including the Gaussian AR(1) process.
}
\\ \\
%%%%%%%%% Key words %%%%%%%%%%%%%%%%%%%%%%%%%%
\\ \\
%%%%%%%%% Subject Classifications %%%%%%%%%

%%%%%%%%%%%%%%%%%%%%%%%%%%%%%%%%%%%%%%%%%%%%%%%%%%%%%%%%%%%%%%%%%%%%%%%%%%%%%%%%%%%%%%%%%%%%%%%%%%%%%%%%%%%%%%%%%%%%%%%%%%%%%%%%%%%%%%%%%%%%%%%%%%%%%%
\section{Introduction}
%%%%%%%%%%%%%%%%%%%%%%%%%%%%%%%%%%%%%%%%%%%%%%%%%%%%%%%%%%%%%%%%%%%%%%%%%%%%%%%%%%%%%%%%%%%%%%%%%%%%%%%%%%%%%%%%%%%%%%%%%%%%%%%%%%%%%%%%%%%%%%%%%%%%%%

The treatment of sequential analysis, in general, and sequential hypothesis testing, in particular, can be roughly divided into two main strands. The
first one was started by \citet{Wald1947} in his pioneering work on sequential analysis. Without being aware of it, Wald used fundamental
properties of martingales to establish bounds on the error probabilities and the expected run-length of sequential tests between two simple
hypotheses. In the decades after Wald's initial publication, research on random processes made significant progress. Theories of martingales, renewal,
and L\'evy processes emerged and became powerful tools in sequential analysis \citep{Lai1977,Lai2009,Buonaguidi2013}. They allowed for elegant
derivations of asymptotically optimal procedures on the one hand, and bounds or approximations on non-asymptotic performance measures on the other.
Following Wald's footsteps, many of the theoretical findings obtained this way resulted in elementary guidelines for the design of sequential tests
that were readily applicable in fields as diverse as survival analysis \citep{Sellke1983}, radar sensing \citep{Marcus1962}, and image processing
\citep{Basseville1981}. For the design of strictly optimal procedures, however, a conceptually different approach to the problem proved to be more
successful.

The second strand of sequential analysis was started by \citet{Bellman1954} with his publications on dynamic programming. Coming from a background in
physics and computer science, Bellman sought to develop a theory ``to treat the mathematical problems arising from the study of various multi-stage
decision processes'', which can be found in ``virtually every phase of modern life, from the planning of industrial production lines to the scheduling
of patients at a medical clinic [\ldots]'' \citep{Bellman1954}. Sequential hypothesis testing, however, had not been added to this list before the
1960s, when \citet{Chow1963} embedded a general theory of optimal stopping in the dynamic programming framework. Treating a sequential test not as a
single threshold crossing problem, but as a sequence of individual decisions to either stop or continue the test, added significantly to the general
insight into sequential inference problems and made the theoretical derivation of strictly optimal methods possible \citep{Chow1971, Novikov2009}.

Naturally, both approaches have coalesced over the years and nowadays sequential testing is usually treated as an optimal stopping problem
that exploits certain stochastic properties of the underlying random process---see \citet{Buonaguidi2013} for a recent example. However, to the
present day, the use of dynamic programming methods to design sequential tests is uncommon in practice. Most often, optimal stopping theory is
rather used to prove the existence of an optimal rule or to derive its general form. Wald's sequential probability ratio test is a classic example.
While optimal stopping theory provides a seamless and elegant way to show its optimality in the i.i.d.\ case \citep{Shiryaev1978}, obtaining the exact
values of the thresholds is much harder a task and requires considerable computational effort.

The reason for this high computational cost is twofold: First, the Bellman equation of the optimal stopping problem itself needs to be solved.
Traditionally, and owing to its roots in dynamic programming, iterative backward recursion procedures \citep{Bellman1954} are used for this purpose.
Under certain conditions, more efficient methods are applicable \citep{Helmes2002}, including a number of linear programming (LP) techniques
\citep{Manne1960, Denardo1979, Roehl2001}. The second issue that arises when deriving optimal strategies is that the sequential testing problem can
not be tackled directly, but has to be reformulated in order to fit the optimal stopping framework. This involves the introduction of two cost
coefficients that determine the price for making an error of either type. The optimal stopping strategy depends on the choice of these coefficients.
The design of sequential tests in an optimal stopping framework, therefore, requires a second, outer optimization over the unknown cost coefficients.
These intricacies limit the design and application of strictly optimal decision strategies.

The aim of this paper is to overcome these issues, by unifying the stopping problem and the problem of choosing the right cost coefficients. The 
formulation of sequential testing as an optimal stopping problem is briefly addressed in Section \ref{sec:seq_det}, where also the Bellman equation
that characterizes its solution is stated. In Section \ref{sec:properties}, this solution is studied in detail and a connection between its derivative
with respect to the cost coefficients and the error probabilities of the sequential test is derived. Based on these results, the original sequential
testing problem is then addressed in Section \ref{sec:linear_program}, where it is shown, under certain Markov restrictions, that it can be formulated
as a jointly linear optimization over the cost coefficients and the optimal stopping strategy. This constitutes the main theoretical contribution of
the work. Since linear programming is rarely ever considered in the context of sequential inference, establishing this connection can be seen as
result in its own right. In addition, the main theorem is of high practical relevance since it allows a large class of sequential testing problems to
be treated within one consistent framework and solved with off-the-shelf linear programming algorithms. This procedure is demonstrated in Section
\ref{sec:examples}, where optimal sequential tests are designed for i.i.d.\ observations, an observable Markov chain and the Gaussian AR(1) process.

A note on notation: $\mathbb{R}_+$ denotes the positive reals and $\mathcal{B}_+$ the associated Borel $\sigma$-algebra. Random variables are denoted
by upper-case letters, their realizations by the corresponding lower-case letters. All (in)equalities between tuples have to be read element-wise.

%%%%%%%%%%%%%%%%%%%%%%%%%%%%%%%%%%%%%%%%%%%%%%%%%%%%%%%%%%%%%%%%%%%%%%%%%%%%%%%%%%%%%%%%%%%%%%%%%%%%%%%%%%%%%%%%%%%%%%%%%%%%%%%%%%%%%%%%%%%%%%%%%%%%%%
\section{Sequential Tests for Two Simple Hypotheses}
\label{sec:seq_det}
%%%%%%%%%%%%%%%%%%%%%%%%%%%%%%%%%%%%%%%%%%%%%%%%%%%%%%%%%%%%%%%%%%%%%%%%%%%%%%%%%%%%%%%%%%%%%%%%%%%%%%%%%%%%%%%%%%%%%%%%%%%%%%%%%%%%%%%%%%%%%%%%%%%%%%
The problem of sequentially testing between two simple hypotheses, under different restrictions and assumptions, has been treated extensively in the
literature; see, for example, \citet{Tartakovsky2014} and references therein. The purpose of this section is to give a summary of known results,
as well as to present them in a form that facilitates the derivations in the subsequent sections. It closely follows \citet{Novikov2009} in terms of
argumentation and notation.

Let $(X_n)_{n \geq 1}$, with $n \in \mathbb{N}$, be a sequence of random variables in a metric state space $(E_X,\mathcal{E}_X)$, defined on some
filtered probability space $\left(\Omega, \mathcal{F}, (\mathcal{F}_n)_{n\geq0}, P \right)$. The two simple hypotheses are given by
\begin{align*}
  \mathcal{H}_0: P & = P_0, \\
  \mathcal{H}_1: P & = P_1.
\end{align*}
Kolmogorov's consistency theorem \citep{Kallenberg1997} states that $P$ is uniquely defined by the marginal distributions of all subsequences $X_1,
\ldots, X_n$, $n \geq 1$. Let these distributions be denoted $F_0^n$ and $F_1^n$ under $\mathcal{H}_0$ and $\mathcal{H}_1$, respectively. The classic
sequential testing problem is to design a test that guarantees certain error probabilities under $P_0$ and $P_1$ and minimizes the run-length of the
test under some third measure $P$, corresponding to a sequence of marginal distributions $(F^n)_{n\geq1}$. In general, $P$ can be chosen arbitrarily;
common choices are $P= P_0$ or $P = P_1$, which correspond to run-lengths minimizations under either hypothesis.

%----------------------------------------------------------------------------------------------------------------------------------------------------%
\subsection{Assumptions}
%----------------------------------------------------------------------------------------------------------------------------------------------------%
The framework proposed in this paper covers processes $(X_n)_{n \geq 1}$ that satisfy the following three assumptions. A bullet $\bullet$ is used
to indicate that an assumption holds under $P$, $P_0$ and $P_1$.

\begin{enumerate}
  \item $(X_n)_{n\geq1}$ admits a time-homogeneous Markovian representation. This means that a sequence of sufficient statistics 
        $\left(\Theta_n\right)_{n\geq0}$ in a state space $(E_{\theta},\mathcal{E}_{\theta})$ exists such that
        \begin{equation*}
          P_{\bullet}(X_{n+1} \in B \,|\, \mathcal{F}_n) = P_{\bullet}(X_{n+1} \in B \,|\, \Theta_n) =: F_{\bullet,\Theta_n}(B)
        \end{equation*}
        for all $n \geq 0$ and all $B \in \mathcal{E}_X$. It is further assumed that a $P$-measurable function $\xi: E_X \times E_{\theta} \rightarrow
        E_{\theta}$ exists such that
        \begin{equation*}
          \Theta_n = \xi(X_n,\Theta_{n-1}) =: \xi_{\Theta_{n-1}}(X_n) \quad \text{with} \quad \Theta_0 = \theta_0.
        \end{equation*}
        In words, the distribution of $X_{n+1}$ is completely specified by $\theta_n$, which can in turn be calculated recursively from $\theta_{n-1}$
        and the observation $x_n$. The initial value $\theta_0$ is deterministic and given a priori. Random sequences admitting these properties cover
        a wide range of commonly used models, such as ARMA and ARCH models and general Markov chains. In order to avoid technical difficulties,
        $E_{\theta}$ is assumed to be Borelian.

  \item The density functions $f^n$,$f_0^n$, and $f_1^n$, corresponding to $F^n$,$F_0^n$, and $F_1^n$, exist with respect to some common product
        measure $\mu^n = \mu(x_1) \otimes \cdots \otimes \mu(x_n)$ and can, according to the first assumption, be written as
        \begin{equation*}
          f_{\bullet}^n(x_1, \ldots, x_n) = \prod_{k=1}^n f_{\bullet}(x_k|\theta_{k-1}) =: \prod_{k=1}^n f_{\bullet,\theta_{k-1}}(x_k).
        \end{equation*}
        for all $n \geq 1$.

  \item The Radon-Nikodym derivatives, or likelihood ratios,
        \begin{equation*}
          z_i^n \;:\; E_X^n \to \mathbb{R}_+ \quad \text{with} \quad z_i^n = \frac{\dint F_i^n}{\dint F^n} =
          \prod_{k=1}^n \frac{f_{i,\theta_{k-1}}(x_k)}{f_{\theta_{k-1}}(x_k)}, \quad i=0,1
        \end{equation*}
        are jointly continuous random variables. Note that this implies that $F_0^n$ and $F_1^n$ are dominated by $F^n$ for all $n\geq1$. The reasons
        for this particular choice of assumptions will become apparent in the course of the paper.
\end{enumerate}

%----------------------------------------------------------------------------------------------------------------------------------------------------%
\subsection{Constrained Problem Formulation}
%----------------------------------------------------------------------------------------------------------------------------------------------------%
Let the sequence $\psi = \left(\psi_n\right)_{n \geq 1}$, with $\psi_n = \psi_n(x_1,\ldots,x_n) \in \{0,1\}$, define the stopping rule of the
sequential test. Here $\psi_n = 1$ denotes the decision to stop at time instant $n$ and $\psi_n = 0$ denotes the decision to continue testing.
Analogously, let $\phi=\left(\phi_n\right)_{n \geq 1}$, with $\phi_n = \phi_n(x_1,\ldots,x_n) \in \{0,1\}$, be a sequence of decision rules, where
$\phi_n = j$ indicates a decision for hypothesis $j$, given that the test stops at time $n$. It is shown later that the restriction of continuously
distributed likelihood ratios avoids the need for randomized stopping and decision rules.

Let $\tau = \tau(\psi)$ be the stopping time that corresponds to the stopping rule $\psi$, i.e.,
\begin{align*}
  \tau = \min \{ n \geq 0 \;:\; \psi_n = 1 \}.
\end{align*}
The error probabilities of the first and second kind, $\alpha_0$ and $\alpha_1$, are given by
\begin{align*}
  \alpha_0(\psi, \phi) & = P_0[\phi_{\tau} = 1], \\ 
  \alpha_1(\psi, \phi) & = P_1[\phi_{\tau} = 0]. 
\end{align*}
The sequential testing problem is to solve
\begin{equation}
  \min_{\psi,\phi} \; E[\tau(\psi)] \quad \text{s.t.} \quad \alpha_i(\psi, \phi) \leq \gamma_i, \quad i=0,1,
  \label{eq:constr_min}
\end{equation}
where $\gamma = (\gamma_0,\gamma_1) \in (0,1)^2 := (0,1) \times (0,1)$ are bounds on the error probabilities, or target error probabilities, and the
expected value is taken with respect to $P$. The solution of \eqref{eq:constr_min} is denoted $(\psi^*_{\gamma},\phi^*_{\gamma})$.

%----------------------------------------------------------------------------------------------------------------------------------------------------%
\subsection{Unconstrained Problem Formulation}
%----------------------------------------------------------------------------------------------------------------------------------------------------%
The commonly used procedure to solve \eqref{eq:constr_min}, in an optimal stopping framework, is to reformulate the constrained problem as an
unconstrained cost minimization by showing that two constants $\lambda_0, \lambda_1 > 0$ exist such that the solution of
\begin{equation}
  \min_{\psi,\phi} \; E[\tau(\psi)] + \lambda_0 \alpha_0(\psi, \phi) + \lambda_1 \alpha_1(\psi, \phi),
  \label{eq:cost_min}
\end{equation}
in the following denoted $(\psi^*_{\lambda},\phi^*_{\lambda})$, coincides with the solution of \eqref{eq:constr_min}, i.e.,
\begin{equation*}
  \forall \gamma \in (0,1)^2 \; : \quad \exists \lambda \in \mathbb{R}_+^2 \; : \quad (\psi^*_{\gamma},\phi^*_{\gamma}) =
(\psi^*_{\lambda},\phi^*_{\lambda}).
\end{equation*}
This method has been used early on to prove optimality of the sequential probability ratio test for i.i.d.\ observations \citep{Wald1948}; a general
proof can be found in \citet{Novikov2009}. The cost coefficients $\lambda_0$ and $\lambda_1$ in \eqref{eq:cost_min} act as Lagrangian multipliers.
However, the question how to choose them in order to meet the constraints on the error probabilities in \eqref{eq:constr_min} has received little to
no attention in the literature. It is discussed in detail in the next section.

For now, $\lambda$ is assumed to be given and fixed. Following the usual line of arguments, the second part of the objective function in
\eqref{eq:cost_min} can be written as
\begin{align*}
  \lambda_0 \alpha_0(\psi, \phi) + \lambda_1 \alpha_1(\psi, \phi) & = \sum_{n \geq 1} \left( \lambda_0 E_0[\psi_n \boldsymbol{1}_{\{\phi_n = 1\}}]
    + \lambda_1 E_1[\psi_n \boldsymbol{1}_{\{\phi_n = 0\}}] \right) \\
  & = \sum_{n \geq 1} \int \psi_n ( \lambda_0 f_0^n \boldsymbol{1}_{\{\phi_n = 1\}} + \lambda_1 f_1^n \boldsymbol{1}_{\{\phi_n = 0\}}) \, \dint \mu^n.
\end{align*}
The cost minimizing decision rules $\phi_{\lambda}^*$ are hence given by
\begin{equation} 
  \phi^*_{\lambda,n} = \boldsymbol{1}_{\{\lambda_0 f_0^n \leq \lambda_1 f_1^n\}}, \quad n \geq 1,
  \label{eq:decision_rule}
\end{equation}
where
\begin{equation*}
  \{\lambda_0 f_0^n \leq \lambda_1 f_1^n\} := \{ (x_1,\ldots,x_n) \in E_X^n \,:\, \lambda_0 f_0^n(x_1,\ldots,x_n) \leq \lambda_1 f_1^n(x_1,\ldots,x_n)
    \}.
\end{equation*}
This shorthand notation is used for subsets of the state space throughout the paper. The choice in \eqref{eq:decision_rule} that the ambiguous
event $\{\lambda_0 f_0^n = \lambda_1 f_1^n\}$ leads to a decision for $\mathcal{H}_0$ is arbitrary since $P(\{\lambda_0 f_0^n = \lambda_1 f_1^n\}) =
0$ for all $n\geq1$ by Assumption 3.

The decision rule \eqref{eq:decision_rule} corresponds to a classic likelihood ratio test with threshold $\lambda_0 / \lambda_1$. Knowledge of the
likelihood ratio $f_1^n / f_0^n$ is therefore sufficient to decide for a hypothesis once the test has stopped. The optimal stopping strategy, however,
requires additional information as will become apparent later on.

Substituting \eqref{eq:decision_rule} into \eqref{eq:cost_min} yields
\begin{align*}
  E[\tau(\psi)] + \lambda_0 \alpha_0(\psi, \phi_{\lambda}^*) + \lambda_1 \alpha_1(\psi, \phi_{\lambda}^*) & = \sum_{n \geq 1} \int \psi_n ( n f^n +
    \min \{ \lambda_0 f_0^n, \lambda_1 f_1^n\}) \, \dint \mu^n \\
  & = \sum_{n \geq 1} \int \psi_n (n + \min \{ \lambda_0 z_0^n, \lambda_1 z_1^n\}) f^n \, \dint \mu^n \\
  & = \sum_{n \geq 1} E\left[ \psi_n \left(n + g_{\lambda}(z^n)\right) \right],
\end{align*}
where $z^n = (z_0^n,z_1^n)$ and
\begin{equation*}
  g_{\lambda}(z) = \min\{ \lambda_0 z_0, \lambda_1 z_1 \}.
\end{equation*}
Problem \eqref{eq:cost_min} therefore reduces to the optimal stopping problem
\begin{equation}
  \label{eq:stop_prob}
  \min_{\psi} \; V_{\lambda}(\psi) := \sum_{n \geq 1} E\left[ \psi_n \left(n + g_{\lambda}(z^n)\right) \right]
\end{equation}
with recursively defined state variables
\begin{equation*}
  z_i^n = \frac{f_i^n}{f^n} =  z^{n-1}_i \frac{f_{i,\theta_{n-1}}}{f_{\theta_{n-1}}}, \quad i=0,1,
\end{equation*}
and $z_i^0 = 1$. The solution of \eqref{eq:stop_prob} is given in the following theorem.

\begin{theorem} \label{th:optimal_stopping}
  Let $\lambda > 0$ be given. Under the three assumptions stated above, the functional equation
  \begin{equation} \label{eq:Wald_Bellman_explicit}
    \rho_{\lambda}(z,\theta) = \min \left\{ g_{\lambda}(z) \; , \; 1 + \int \rho_{\lambda}\left( z_0 \frac{f_{0,\theta}(x)}{f_{\theta}(x)}, z_1
      \frac{f_{1,\theta}(x)}{f_{\theta}(x)}, \xi_{\theta}(x) \right) \, \dint F_{\theta}(x) \right\}
  \end{equation}
  has a unique solution $\rho_{\lambda} \geq 0$ on $(E,\mathcal{E}) := (\mathbb{R}_+^2 \times E_{\theta}, \mathcal{B}_+^2 \otimes
  \mathcal{E}_{\theta})$ and it holds that
  \begin{equation*}
    \rho_{\lambda}(1,1,\theta_0) = V_{\lambda}(\psi_{\lambda}^*).
  \end{equation*}
\end{theorem}

A proof of theorem \ref{th:optimal_stopping} is given in Appendix \ref{apd:proof_optimal_stopping}. By a change in measure,
\eqref{eq:Wald_Bellman_explicit} can alternatively be written as
\begin{equation}
  \rho_{\lambda}(z,\theta) = \min \left\{ g_{\lambda}(z) \; , \; 1 + \int \rho_{\lambda} \, \dint H_{z,\theta} \right\},
  \label{eq:Wald_Bellman_modified}
\end{equation}
where $\{H_{z,\theta} : (z,\theta) \in E\}$ is a family of probability measures on $(E,\mathcal{E})$ satisfying
\begin{align}
  H_{z,\theta}(B \times D) = F_{\theta} \left( \left\{ x \in E_X \,:\, \left(z_0 \frac{f_{0,\theta}(x)}{f_{\theta}(x)}, z_1
    \frac{f_{1,\theta}(x)}{f_{\theta}(x)}\right) \in B, \; \xi_{\theta}(x) \in D \right\} \right)
  \label{eq:definition_H}
\end{align}
for all $B \in \mathcal{B}_+^2$ and $D \in \mathcal{E}_{\theta}$. Use of both formulations is made in the following. The notation
$H^i_{z,\theta}$, $i=0,1$, is used to refer to the families of distributions where $F$ in \eqref{eq:definition_H} is replaced by $F_i$. Note that
$g_{\lambda}$ is $H_{z,\theta}$-integrable for all $\lambda \in \mathbb{R}_+^2$ and $(z,\theta) \in E$ since
\begin{equation*}
  \int g_{\lambda} \, \dint H_{z,\theta} \leq \int \lambda_0 z_0 \frac{f_{i,\theta}}{f_{\theta}} \, \dint F_{\theta} = \lambda_0 z_0 < \infty.
\end{equation*}
Therefore, $\rho_{\lambda} \leq g_{\lambda}$ is $H_{z,\theta}$-integrable as well.

The importance of Theorem \ref{th:optimal_stopping} lies in the fact that it allows the optimal stopping region to be specified on the codomain
of $(z^n,\theta_n)$, which is time invariant, rather than the sequence of product spaces $(E_X^n)_{n\geq1}$ corresponding to the raw observations. Let
the stopping region $\mathcal{S}_{\lambda}$, its boundary $\partial\mathcal{S}_{\lambda}$, and its complement $\overline{\mathcal{S}}_{\lambda}$ be
defined by
\begin{align*}
  \mathcal{S}_{\lambda} & = \left\{ (z,\theta) \in E \,:\, g_{\lambda}(z) \; < \; 1 + \int
    \rho_{\lambda} \, \dint H_{z,\theta} \right\}, \\
  \partial\mathcal{S}_{\lambda} & = \left\{ (z,\theta) \in E \,:\, g_{\lambda}(z) \; = \; 1 + \int
    \rho_{\lambda} \, \dint H_{z,\theta} \right\}, \\
  \overline{\mathcal{S}}_{\lambda} & = \left\{ (z,\theta) \in E \,:\, g_{\lambda}(z) \; > \; 1 + \int
    \rho_{\lambda} \, \dint H_{z,\theta} \right\}.
\end{align*}
The connection between the optimal stopping policy $\psi^*$ and the functional equation in Theorem \ref{th:optimal_stopping} is stated in the
following corollary.
\begin{corollary} \label{cl:stopping_region}
  Under the assumptions given above, the stopping policy that minimizes $V_{\lambda}$ in \eqref{eq:stop_prob} fulfills
  \begin{equation*}
    \boldsymbol{1}_{\mathcal{S}_{\lambda}} \leq \psi_{\lambda,n}^* \leq \boldsymbol{1}_{\mathcal{S}_{\lambda} \cup \partial\mathcal{S}_{\lambda}}
    \quad \forall n \geq 1.
  \end{equation*}
\end{corollary}
The corollary follows immediately from the the definition of $\rho_{\lambda}$, cf.\ \cite{Novikov2009}. The time-in\-va\-ri\-ance of the stopping
strategy reflects the time-homogeneity of the Markov sequence underpinning the test. In Section \ref{sec:properties} it is further shown that under
the given assumptions $\partial\mathcal{S}_{\lambda}$ is a $P$ null set so that the relations in Corollary \ref{cl:stopping_region} hold with equality
in an almost sure sense.

Theorem \ref{th:optimal_stopping} and, in particular, the implicit definition of the cost function $\rho_{\lambda}$ in
\eqref{eq:Wald_Bellman_explicit}, provides the basis for formulating the sequential detection problem as a linear program. This formulation requires
several properties of the cost function $\rho_{\lambda}$, which are detailed in the following section.

%%%%%%%%%%%%%%%%%%%%%%%%%%%%%%%%%%%%%%%%%%%%%%%%%%%%%%%%%%%%%%%%%%%%%%%%%%%%%%%%%%%%%%%%%%%%%%%%%%%%%%%%%%%%%%%%%%%%%%%%%%%%%%%%%%%%%%%%%%%%%%%%%%%%%%
\section{Properties of the Cost Function $\boldsymbol{\rho_{\lambda}}$}
\label{sec:properties}
%%%%%%%%%%%%%%%%%%%%%%%%%%%%%%%%%%%%%%%%%%%%%%%%%%%%%%%%%%%%%%%%%%%%%%%%%%%%%%%%%%%%%%%%%%%%%%%%%%%%%%%%%%%%%%%%%%%%%%%%%%%%%%%%%%%%%%%%%%%%%%%%%%%%%%
In this section, some basic properties of $\rho_{\lambda}$ are derived and the close relationship between its derivatives with respect to $\lambda$
and the error probabilities of the corresponding sequential test are shown. Some technical issues need to be addressed first.

\begin{lemma} \label{lm:uniform_convergence}
  The sequence $(\rho_{\lambda}^n)_{n\geq0}$ with $\rho_{\lambda}^n = T^n(g_{\lambda})$ and $T$ defined in \eqref{eq:def_T}, converges uniformly on
  $E$.
\end{lemma}
The proof of Lemma \ref{lm:uniform_convergence} is detailed in Appendix \ref{apd:proof_uniform_convergence}.

Sequences of functions of the same structure as $(\rho_{\lambda}^n)_{n\geq0}$, where the $n$th function is defined recursively via an integration over
the previous function, are encountered repeatedly in this section. The same arguments as detailed in Appendix \ref{apd:proof_uniform_convergence} can
be used to show that in such cases almost uniform convergence implies uniform convergence. Therefore, instead of giving a formal proof again, we refer
to Lemma \ref{lm:uniform_convergence} in what follows. %FIRSTPERSON?%

The next lemma bounds the influence of the likelihood ratio on the associated cost. Qualitatively speaking, it states that $\rho_{\lambda}$ is
monotonic and sublinear in $z$.
\begin{lemma} \label{lm:rho_bounds}
  For all $a,z \in \mathbb{R}_+^2$ and $\theta \in E_{\theta}$, the cost function $\rho_{\lambda}$ as defined in \eqref{eq:Wald_Bellman_explicit}
  satisfies
  \begin{equation*}
    \min\{a_0,a_1,1\} \rho_{\lambda}(z,\theta) \, \leq \, \rho_{\lambda}(a_0 z_0, a_1 z_1,\theta) \, \leq \, \max\{a_0,a_1,1\}
    \rho_{\lambda}(z,\theta).
  \end{equation*}
\end{lemma}
\begin{proof}
  We only prove the lower bound since the upper one can be shown analogously. Assume that the lemma holds for some $\rho_{\lambda}^n$ and let $a^* =
  \min\{a_0,a_1,1\}$. By induction it holds that
  \begin{align*}
    \rho^{n+1}_{\lambda}(a_0 z_0, a_1 z_1,\theta) & = \min \left\{ g_{\lambda}(a_0 z_0, a_1 z_1) \,,\, 1+\int \rho^n_{\lambda}\left(
      a_0 z_0 \frac{f_{0,\theta}}{f_{\theta}}, a_1 z_1 \frac{f_{1,\theta}}{f_{\theta}}, \xi_{\theta} \right) \dint F_{\theta} \right\} \\
    & \geq \min \left\{ a^* g_{\lambda}(z) \,,\, a^* + a^* \int \rho^n_{\lambda}\left(z_i \frac{f_{0,\theta}}{f_{\theta}}, z_1
      \frac{f_{1,\theta}}{f_{\theta}}, \xi_{\theta} \right) \dint F_{\theta} \right\} \\
    & = a^* \rho^{n+1}_{\lambda}(z,\theta).
  \end{align*}
  The induction basis is $g_{\lambda}(a_0 z_0, a_1 z_1) \geq a^* g_{\lambda}(z)$.  By \citet{Rudin1987}, uniform convergence of $(\rho^n_{\lambda})_{n
  \geq 1}$ guarantees that the bound holds for $\rho_{\lambda}$ as well.
\end{proof}

Using the bounds in Lemma \ref{lm:rho_bounds}, it is easy to show that the boundary $\partial\mathcal{S}_{\lambda}$ is indeed a $P$ null set.
\begin{lemma} \label{lm:null_set}
  If the likelihood ratios $z^n$ are continuous random variables for all $n \geq 1$, the boundary $\partial\mathcal{S}_{\lambda}$ of the optimal
  stopping region is a null set under $P$, i.e.,
  \begin{align*}
    H_{z,\theta}(\partial\mathcal{S}_{\lambda}) = 0 \quad \forall (z,\theta) \in E.
  \end{align*}
\end{lemma}
\begin{proof}
  Lemma \ref{lm:null_set} can be shown by contradiction. Assume that $H_{z,\theta}(\partial\mathcal{S}_{\lambda}) > 0$ for some $(z,\theta)$, i.e.,
  assume that if $(z^n,\theta_n) = (z,\theta)$, then there is a nonzero probability that the test hits $\partial\mathcal{S}_{\lambda}$ with the next
  update of the test statistic. Since  all $z^n$ are assumed to be continuous random variables, this implies that an interval $[az^*, z^*]$, $a < 1$,
  exist on which
  \begin{align*}
    1 + \int \rho_{\lambda} \dint H_{z,\theta} = g_{\lambda}(z)
  \end{align*}
  for all $z \in [az^*, z^*]$. However, by Lemma \ref{lm:rho_bounds},
  \begin{align*}
    1 + \int \rho_{\lambda} \, \dint H_{az^*,\theta} & \geq 1 + \int a \rho_{\lambda}\left( z_0 \frac{f_{0,\theta}}{f_{\theta}}, z_1
      \frac{f_{1,\theta}}{f_{\theta}}, \xi_{\theta} \right) \, \dint F_{\theta} \\
    & > a \left(1 + \int  \rho_{\lambda}\left( z_0 \frac{f_{0,\theta}}{f_{\theta^*}}, z_1 \frac{f_{1,\theta}}{f_{\theta}}, \xi_{\theta}
      \right) \, \dint F_{\theta} \right) \\
    & = a \left( 1 + \int \rho_{\lambda} \, \dint H_{z^*,\theta} \right) \\
    & = a g_{\lambda}(z^*) = g_{\lambda}(az^*),
  \end{align*}
  which contradicts the assumption.
\end{proof}

Lemma \ref{lm:null_set} implies that the optimal sequential test is non-randomized since the cost minimizing decision is almost surely unambiguous.

It is now possible to give expressions for the derivatives of $\rho_{\lambda}$.

\begin{theorem} \label{th:derivative}
  Let  $\rho'_{\lambda_i}$ denote the derivative of $\rho_{\lambda}$ with respect to $\lambda_i$, $i=0,1$. Further, define
  \begin{align*}
    \mathcal{S}_{\lambda_i} := \mathcal{S}_{\lambda} \cap \{ \lambda_i z_i < \lambda_{1-i}z_{1-i}\}.
  \end{align*}
  The two Fredholm integral equations of the second kind
  \begin{align} \label{eq:derivative_lemma_eq}
    r_{\lambda_i}(z,\theta) = z_i H^i_{z,\theta}\left(\mathcal{S}_{\lambda_i}\right) + \int_{\overline{\mathcal{S}}_{\lambda}} r_{\lambda_i} \,
      \dint H_{z,\theta}, \quad i=0,1,
  \end{align}
  have a unique solution $r_{\lambda} = (r_{\lambda_0},r_{\lambda_1})$ on $\overline{\mathcal{S}}_{\lambda}$ and it holds that
  \begin{align*}
    \rho'_{\lambda_i}(z,\theta) = \begin{cases}
                                    z_i, & (z,\theta) \in \mathcal{S}_{\lambda_i} \\
                                    r_{\lambda_i}(z,\theta), & (z,\theta) \in \overline{\mathcal{S}}_{\lambda} \\
                                    0, & (z,\theta) \in \mathcal{S}_{\lambda_{1-i}}.
                                  \end{cases}
  \end{align*}
\end{theorem}
A proof of Theorem \ref{th:derivative} can be found in Appendix \ref{apd:proof_derivative}.

Based on Theorem \ref{th:derivative}, the main result of this section can be shown, which is the connection between the cost coefficients $\lambda$
and the error probabilities of the corresponding cost-minimizing sequential test.
\begin{theorem} \label{th:error_prob}
  For $\rho'_{\lambda_i}$, as defined in Theorem \ref{th:derivative}, it is the case that
  \begin{equation*}
    \rho'_{\lambda_i}(z,\theta) = z_i P_{i}[\phi^*_{\lambda,\tau} = 1-i \,|\, z^0 = z, \theta_0 = \theta], \quad i=0,1
  \end{equation*}
  on $E \setminus \partial \mathcal{S}_{\lambda}$ and in particular
  \begin{equation*}
    \rho'_{\lambda_i}(1,1,\theta_0) = \alpha_i(\phi_{\lambda}^*,\psi_{\lambda}^*).
  \end{equation*}
\end{theorem}

See Appendix \ref{apd:proof_error_prob} for a proof of Theorem \ref{th:error_prob}.

It is worth mentioning that $(1,1,\theta_0)$ is an element of $\overline{\mathcal{S}}_{\lambda}$ for every meaningful choice of $\lambda$ since
otherwise the optimal strategy would be not to take any samples at all, but to decide on a hypothesis a priori, cf.\ \cite{Novikov2009}. Such a
trivial strategy can indeed be cost minimizing, if $\lambda$ is chosen sufficiently small. However, by only considering target error probabilities
from the open unit interval, i.e., $\gamma_i \in (0,1)$, trivial tests are excluded from the set of possible solutions of the constrained problem.

Theorem \ref{th:error_prob} is central to this work since it directly connects the solution of the optimal stopping problem \eqref{eq:stop_prob} to
the error probabilities of the corresponding sequential test. Exploiting this connection is key to converting the sequential testing problem
\eqref{eq:constr_min} into a linear program.

%%%%%%%%%%%%%%%%%%%%%%%%%%%%%%%%%%%%%%%%%%%%%%%%%%%%%%%%%%%%%%%%%%%%%%%%%%%%%%%%%%%%%%%%%%%%%%%%%%%%%%%%%%%%%%%%%%%%%%%%%%%%%%%%%%%%%%%%%%%%%%%%%%%%%%
\section{Sequential Testing as a Linear Program}
\label{sec:linear_program}
%%%%%%%%%%%%%%%%%%%%%%%%%%%%%%%%%%%%%%%%%%%%%%%%%%%%%%%%%%%%%%%%%%%%%%%%%%%%%%%%%%%%%%%%%%%%%%%%%%%%%%%%%%%%%%%%%%%%%%%%%%%%%%%%%%%%%%%%%%%%%%%%%%%%%%
In this section we derive, based on the results in Sections \ref{sec:seq_det} and \ref{sec:properties}, the linear form of the sequential hypothesis
testing problem stated in \eqref{eq:constr_min}.

Similar to the approach in Section \ref{sec:seq_det}, the first step to a linear program is to include the constraints in \eqref{eq:constr_min} in the
cost function. However, the cost coefficients here are not chosen a priori, but are instead introduced by directly applying the machinery of
Lagrangian duality to the problem.

The Lagrangian dual problem of \eqref{eq:constr_min} is given by
\begin{equation}
  \max_{\lambda > 0} \; L_{\gamma}(\lambda),
  \label{eq:dual_problem}
\end{equation}
where
\begin{align*}
  L_{\gamma}(\lambda) & = \min_{\psi,\phi} \; \left\{ E[\tau(\psi)] + \lambda_0 (\alpha_0(\psi, \phi) - \gamma_0) + \lambda_1 (\alpha_1(\psi,
      \phi)-\gamma_1) \right\} \\
  & = \min_{\psi,\phi} \left\{ E[\tau(\psi)] + \lambda_0 \alpha_0(\psi, \phi) + \lambda_1 \alpha_1(\psi, \phi) \right\} - \lambda_0 \gamma_0 -
      \lambda_1 \gamma_1 \\
  & = \min_{\psi} V_{\lambda}(\psi) - \lambda_0 \gamma_0 - \lambda_1 \gamma_1 \\
  & = V_{\lambda}(\psi^*) - \lambda_0 \gamma_0 - \lambda_1 \gamma_1 \\
  & = \rho_{\lambda}(1,1,\theta_0) - \lambda_0 \gamma_0 - \lambda_1 \gamma_1.
\end{align*}
$L_{\gamma}$ is concave in $\lambda$ by construction. However, the equivalence between \eqref{eq:dual_problem} and \eqref{eq:constr_min}, i.e., the
absence of a duality gap, is not obvious. It is therefore stated explicitly in the following theorem, which is mostly a corollary of the results
stated in the previous sections.

\begin{theorem} \label{th:no_gap}
  Let $\lambda_{\gamma}^*$ be the solution of \eqref{eq:dual_problem}. It then holds that
  \begin{align*}
    (\psi^*_{\lambda_{\gamma}^*},\phi^*_{\lambda_{\gamma}^*}) = (\psi^*_{\gamma},\phi^*_{\gamma}) \quad \text{and} \quad
    L_{\gamma}(\lambda_{\gamma}^*) = E[\tau(\psi^*_{\gamma})],
  \end{align*}
  i.e., the solution of \eqref{eq:dual_problem} coincides with the solution of \eqref{eq:constr_min}.
\end{theorem}

See Appendix \ref{apd:proof_no_gap} for a proof.% of Theorem \ref{th:no_gap}.

By \eqref{eq:dual_problem} and Theorem \ref{th:no_gap}, the original problem \eqref{eq:constr_min} is equivalent to the maximization problem
\begin{align} 
  & \max_{\lambda > 0} \; \rho_{\lambda}(1,1,\theta_0) - \lambda_0 \gamma_0 - \lambda_1 \gamma_1 \label{eq:constrained_problem} \\
  \text{s.t.} \quad & \rho_{\lambda}(z,\theta) = \min \left\{ g_{\lambda}(z) \; , \; 1 + \int \rho_{\lambda} \, \dint H_{z,\theta} \right\}. \notag
\end{align}
The simple trick at this point is to relax the equality constraint to an inequality and add $\rho_{\lambda}$ to the set of free variables. It yields
the main Theorem of this work.

\begin{theorem} \label{th:convex_problem}
  Let $\mathcal{L}$ be the set of nonnegative $H_{z,\theta}$ integrable functions on $E$. The problem
  \begin{align}
  & \max_{\lambda > 0, \, \rho \in \mathcal{L}} \quad \rho(1,1,\theta_0) - \lambda_0 \gamma_0 - \lambda_1 \gamma_1
  \label{eq:convex_problem_max} \\
  \text{s.t.} \quad & \rho(z,\theta) \leq \min \left\{ \lambda_0 z_0 \,,\, \lambda_1 z_1 \,,\, 1 + \int \rho \, \dint H_{z,\theta}
    \right\} \quad \forall (z,\theta) \in E \notag
  \end{align}
  is equivalent to problem \eqref{eq:constr_min}. More precisely, 
  \begin{equation*}
    E[\tau(\psi_{\gamma}^*)] = \rho^*(1,1,\theta_0)- \lambda^*_0 \gamma_0 - \lambda^*_1 \gamma_1, \quad \psi_{\gamma,n}^* =
\boldsymbol{1}_{\left\{g_{\lambda^*}(z^n) = \rho^*(z^n,\theta_n)
      \right\}} \quad \text{and} \quad \phi_{\gamma,n}^* = \boldsymbol{1}_{\left\{\lambda_0^* z_0^n \leq \lambda_1^* z_1^n\right\}},
  \end{equation*}
  where $\lambda^*$ and $\rho^*$ solve \eqref{eq:convex_problem_max}.
\end{theorem}
\begin{proof}
  Qualitatively speaking, the validity of the relaxation in \eqref{eq:convex_problem_max} follows from the fact that every $\rho(z,\theta)$ is a
  nondecreasing function in $\rho(B)$ for all $B \in \mathcal{E}$. Therefore, maximizing $\rho$ at one point implies maximizing $\rho$ over the
  entire state space.

  To formalize this, let $\rho^*$ be the solution of \eqref{eq:constrained_problem} and $\tilde{\rho}$ be the solution of the corresponding relaxed
  problem \eqref{eq:convex_problem_max}. %Owing to the relaxation, $\tilde{\rho}(1,1,\theta_0) \geq \rho^*(1,1,\theta_0)$ in general. 
  Since $\rho^*$ is unique, $\tilde{\rho} = \rho^*$ whenever $\tilde{\rho}$ fulfills the relaxed constraint with equality. Hence, only the case when
  equality does not hold needs closer inspection. In this case, a function $\tilde{\rho}+\Delta \rho$ with
  \begin{equation*}
    \Delta \rho = \min \left\{ g_{\lambda^*} \; , \; 1 + \int \tilde{\rho} \, \dint H_{z,\theta} \right\} - \tilde{\rho} \geq 0
  \end{equation*}
  can be constructed, without changing $\lambda^*$, that still fulfills the inequality constraint, but dominates $\tilde{\rho}$. This procedure can be
  repeated to create a nondecreasing sequence of functions that converges to a solution of the non-relaxed problem \eqref{eq:constrained_problem}.
  Since $\tilde{\rho}$ is assumed to be optimal, this means that $\tilde{\rho} \leq \rho^*$, but $\tilde{\rho}(1,1,\theta_0) = \rho^*(1,1,\theta_0)$.
  For this to hold, $\tilde{\rho}$ and $\rho^*$ must differ only on a $P$ null set. The associated stopping rules are hence equivalent in an almost
  sure sense.
\end{proof}

%----------------------------------------------------------------------------------------------------------------------------------------------------%
\subsection{Discussion}
%----------------------------------------------------------------------------------------------------------------------------------------------------%
Problem \eqref{eq:convex_problem_max} can be written as a generic linear program by splitting the minimum-constraint into three linear inequality
constraints. However, since it involves the optimization over a continuous function, it falls in the class of infinite-dimensional optimization
problems. The solution methods for this kind of problem range from classic calculus of variations \citep{Gelfand2003} to general numerical approaches
\citep{Schochetman2001, Devolder2010} and approaches customized for linear problems \citep{Ito2009}. However, a detailed analysis of
infinite-dimensional optimization techniques is beyond the scope of this work. For the examples presented in Section \ref{sec:examples}, a
straightforward discretization of the problem proved sufficient.

The result of the optimization are optimal cost coefficients $\lambda^*$ and the corresponding cost functions $g_{\lambda^*}$ and $\rho_{\lambda^*}$.
The maximum value of the objective function, i.e., $\rho_{\lambda^*}(1,1,\theta_0)- \lambda_0^* \gamma_0 - \lambda_1^* \gamma_1$, corresponds to the
expected number of samples of the test. The optimal stopping rule is to continue the test as long as $\rho_{\lambda^*}(z^n,\theta_n) <
g_{\lambda^*}(z^n)$ and to stop if $\rho_{\lambda^*}(z^n,\theta_n) = g_{\lambda^*}(z^n)$. Calculating the boundary $\partial \mathcal{S}_{\lambda}$
explicitly is in general not necessary, but can be done to reduce the amount of storage and to compare the optimal test to constant threshold
tests---see Section \ref{sec:examples}.

When solving problem \eqref{eq:convex_problem_max} numerically, it can be the case that on some region $B \subset E$  the inequality constraint is
not fulfilled with equality, even though $B$ is not a $P$ null set. This effect is due to numerical inaccuracies and occurs when the coupling between
$\rho(1,1,\theta_0)$ and $\rho(B)$ is so weak that the contribution of $B$ to $\rho(1,1,\theta_0)$ is smaller than the precision of the solver. As
a result, the stopping region can exhibit some areas, where the cost for continuing is erroneously declared to be smaller than that for stopping.
However, given a reasonable precise solver, these artifacts occur only in regions of the state space that are highly unlikely to ever be reached
during a test and usually are a purely cosmetic problem. In any case, the procedure given in the proof of Theorem \ref{th:convex_problem} can be used
to construct a valid solution from the inaccurate one. Alternatively, a regularization term can be added to the maximization that explicitly enforces
equality---see Appendix \ref{apd:enforce_equality} for details.

An advantage of the LP design approach is that it does not require additional performance analysis. The expected run-length is already a result of
the optimization and the required error probabilities are met exactly, in theory, or within the accuracy of the numerical solver, in practice. This is
in contrast to many design techniques that are two-step procedures: First, a stopping rule or threshold is determined, based on upper bounds on the
error probabilities. Second, the performance of a test using this stopping rule is analyzed. While such approaches work well in the i.i.d.\ case or
the asymptotic case, where limiting distributions and large sample number approximations can be used, they become increasingly involved and inaccurate
for the case of correlated observations or moderate target error probabilities. Hence, the design of sequential tests under such scenarios often
relies on Monte Carlo simulations \citep{Tartakovsky2003} or resampling methods \citep{Sochman2005} to determine the true error probabilities. 

The use of classic numerical optimization methods for the design of sequential tests is uncommon, perhaps because of the hidden nature of the
linearity. The problem is indeed highly nonlinear in the obvious optimization parameters such as the stopping rule or the likelihood ratio
thresholds. More widespread is the practice of evaluating the performance of a given stopping rule by solving the Fredholm equations
\eqref{eq:fredholm_int_errors} numerically \citep{Tartakovsky2014}. However, for the purpose of designing sequential tests, this approach has its
limits since a reasonable estimate of the stopping region has to be known beforehand. While this is unproblematic in the i.i.d.\ case,
constructing stopping regions by hand becomes extremely challenging for the case of correlated observations. The relative simplicity of the linear
programming approach, in contrast, is achieved by avoiding a direct calculation of the stopping region all together. In this sense, it is remarkable
that the boundary manifolds that solve the Fredholm integral equations in \eqref{eq:fredholm_int_errors} can be obtained implicitly by solving
\eqref{eq:convex_problem_max}, whereas performing an explicit optimization with respect to $\partial \mathcal{S}$ is a formidable task.

%%%%%%%%%%%%%%%%%%%%%%%%%%%%%%%%%%%%%%%%%%%%%%%%%%%%%%%%%%%%%%%%%%%%%%%%%%%%%%%%%%%%%%%%%%%%%%%%%%%%%%%%%%%%%%%%%%%%%%%%%%%%%%%%%%%%%%%%%%%%%%%%%%%%%%
\section{Examples and Numerical Results}
\label{sec:examples}
%%%%%%%%%%%%%%%%%%%%%%%%%%%%%%%%%%%%%%%%%%%%%%%%%%%%%%%%%%%%%%%%%%%%%%%%%%%%%%%%%%%%%%%%%%%%%%%%%%%%%%%%%%%%%%%%%%%%%%%%%%%%%%%%%%%%%%%%%%%%%%%%%%%%%%
Three example problems are solved in this section to illustrate the proposed LP approach to sequential detection. The basic task in all of them is to
test for a shift in the mean of a Gaussian random variable. The dependency structures, however, are chosen increasingly complex. A simple i.i.d.\
model is considered first, followed by an observable Markov chain with two states. Finally, the sequential testing problem for the Gaussian AR(1)
process is solved. To the best of our knowledge, the optimal solutions for the two latter models have not been given in the literature before.

The sequential tests in this section are all designed to minimize the run-length under the null hypothesis. The main reason for this choice is that
$z_0 = 1$ for $P = P_0$ so that the optimal test can be performed on $z_1$ only. This significantly simplifies the plots and reduces the technical
difficulties. In addition, minimizing the run-length under a particular hypotheses is a task often encountered in practical problems, like hazard or
fault detection. For notational convenience $z$ is used instead of $z_1$.

The only input to the linear program that is not chosen by the test designer is the family of measures $H_{z,\theta}$. In the following it is
represented by a kernel function $h(z',\theta' \,;\, z,\theta)$ satisfying
\begin{equation*}
  \int_{B} h(z',\theta' \,;\, z,\theta) \, \dint z' \dint \theta' = H_{z,\theta}(B) \quad \forall B \in \mathcal{E}.
\end{equation*}
Note that determining $h$ is the only non-generic step in the test design and therefore the most likely source of errors.

In order to numerically solve problem \eqref{eq:convex_problem_max} all continuous quantities are discretized, including the kernel of the integral
transformation. This corresponds to solving the problem on a grid with a finite number, $M$, of points. The discretized problem in its generic form
reads
\begin{equation}
  \max_{\lambda \in \mathbb{R}_+^2, \, \boldsymbol{\rho} \in \mathbb{R}_+^M} \quad \rho_n - \lambda_0 \gamma_0 - \lambda_1
    \gamma_1 \quad \text{s.t.} \quad \boldsymbol{\rho} \leq \min \left\{ \lambda_0 \boldsymbol{z}_0 \, , \, \lambda_1 \boldsymbol{z}_1 \, , \, 1 +
    \boldsymbol{\rho} \boldsymbol{H} \right\},
\end{equation}
where $\boldsymbol{\rho}, \boldsymbol{z}$ are row vectors of size $M$ and $\boldsymbol{H}$ is a matrix of size $M\times M$ that corresponds to the
integral kernel. The index $n$ is chosen such that $\rho_n$ corresponds to $\rho(1,1,\theta_0)$. For the problems presented here, this straightforward
sampling approach is sufficiently accurate and computationally efficient.

For better numerical stability, however, it is highly recommendable to perform some kind of pre-warping or to use a nonlinear sampling function since
the likelihood ratio values require different sampling granularities on different intervals. For our experiments, we used
\begin{equation*}
  t_{\beta}(z) = \frac{1}{1+z^{-\beta}},
\end{equation*}
where $t: \mathbb{R}_+ \to (0,1]$ maps the positive reals onto the unit interval and $\beta > 0$ can be chosen freely. This mapping can be interpreted
as the concatenation of a logarithmic transform and a logistic transform.

The results in this section are presented in terms of log-likelihood ratio thresholds. Following the majority of the literature, the upper
threshold is denoted $A$ and the lower threshold $B$. In general, the optimal thresholds are functions of the past observations, i.e., $A = A(\theta)$
and $B = B(\theta)$. For the sake of a more compact notation, the cost function associated with continuing the test is denoted
\begin{equation*}
  d_{\lambda}(z,\theta) := 1 + \int \rho_{\lambda} \, \dint H_{z,\theta}.
\end{equation*}
The equivalent cost function for stopping the test is $g_{\lambda}(z)$.

As a reference for comparison with the optimal results, tests are used whose thresholds are calculated according to Wald's approximation, which is
given by
\begin{equation}
  A \approx \log \frac{\alpha_1}{1-\alpha_0}, \quad B \approx \log \frac{\alpha_0}{1-\alpha_1}.
  \label{eq:wald_aprx}
\end{equation}
These approximations are independent of the distributions underpinning the test and are most widely used in practice---irrespective of the fact that
improvements have been suggested throughout the years \citep{Page1954, Tallis1965}. It is further shown in \cite{Wald1948} that the approximations
\eqref{eq:wald_aprx} are asymptotically optimal as $\max\{\gamma_0,\gamma_1\} \to 0$ .

Finally, in this section the quantities marked with a tilde have been obtained by means of Monte Carlo simulations. In all of the experiments,
$10^5$ Monte Carlo runs of the respective sequential tests were performed.

%----------------------------------------------------------------------------------------------------------------------------------------------------%
\subsection{Mean Shifted Gaussian IID}
%----------------------------------------------------------------------------------------------------------------------------------------------------%
The classic problem of ``testing that the mean of a normal distribution with known standard deviation falls short of a given value'' \citep{Wald1947}
is a good example to introduce the LP approach to the design of sequential tests. It can equivalently be formulated as
\begin{align*}
  \mathcal{H}_0: & \quad X_n \sim f_{\mathcal{N}}(x_n \,;\, 0, \sigma), \\
  \mathcal{H}_1: & \quad X_n \sim f_{\mathcal{N}}(x_n \,;\, \mu, \sigma),
\end{align*}
where $f_{\mathcal{N}}(\cdot \,;\, \mu, \sigma)$ denotes the Gaussian probability density function with mean $\mu$ and standard deviation $\sigma$.
All $X_n$ are assumed to be i.i.d.\ under both hypotheses.

Since every observation is independent of the past, $E_{\theta} = \emptyset$. Under $P =P_0$, the log-likelihood ratio $s = \log(z)$ follows a
Gaussian distribution with $\mu_s = -\mu^2/2\sigma^2$ and $\sigma_s = \mu/\sigma$ such that the kernel $h^0$ in terms of $s$ is given by
\begin{equation}
  h^0_{\text{iid}}(s',s \,;\, \mu, \sigma) = f_{\mathcal{N}}(s'-s \,;\, \mu_s, \sigma_s) = \frac{\sigma}{\sqrt{2\pi} \mu} \exp\left(-
    \frac{\sigma^2}{\mu^2}(s'-s) - \frac{1}{2} \right).
  \label{eq:iid_kernel}
\end{equation}

To obtain an expression in $z$, one can simply substitute $s = \log(z)$ in \eqref{eq:iid_kernel}. Alternatively, problem \eqref{eq:convex_problem_max}
can be formulated directly in terms of the log-likelihood ratio as
\begin{align*}
  & \max_{\lambda > 0, \, \rho > 0} \quad \rho(0,0,\theta_0) - \lambda_0 \gamma_0 - \lambda_1 \gamma_1 \\
  \text{s.t.} \quad & \rho(s,\theta) \leq \min \left\{ \lambda_0 e^{s_0} \,,\, \lambda_1 e^{s_1} \,,\, 1 + \int \rho \, \dint H_{s,\theta}
  \right\}
\end{align*}
Expressions in $t_{\beta}(z)$, or any other bijectively transformed version of $z$, can be analogously stated.

For the experiments the parameters $\mu = \sigma = 1$ and $t_{0.5}(z)$ was sampled at 200 equally spaced points. The kernel matrix
$\boldsymbol{H}^0_{\text{iid}}$ is accordingly of dimensions $200 \times 200$. Problems of this size are solved within seconds by state-of-the-art LP
solvers, which makes the algorithm very attractive for the design of sequential tests between two i.i.d.\ sequences.

\begin{table}[t]
  \centering
  \includegraphics{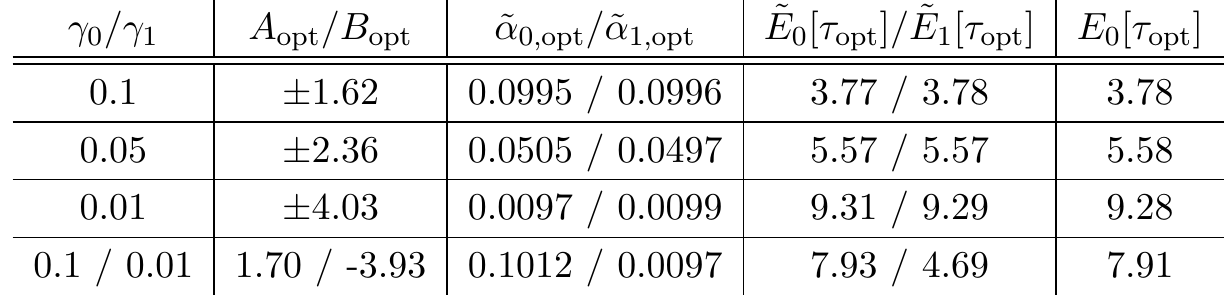}
  \caption{Mean Shifted Gaussian IID: Optimal log-likelihood ratio thresholds $A$, $B$, empirical error probabilities $\tilde{\alpha}_i$, and
           average and expected run-length $\tau$ for target error probabilities $\gamma$}
  \label{tb:iid_opt}
  \includegraphics{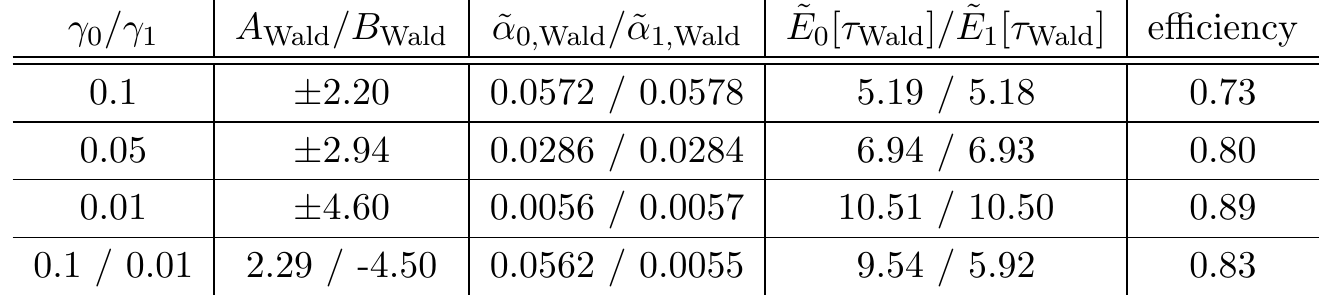}
  \caption{Mean Shifted Gaussian IID: Log-likelihood ratio thresholds $A$, $B$, empirical error probabilities $\tilde{\alpha}_i$ and average
           run-length $\tau$ for target error probabilities $\gamma$ using Wald's approximations. The last column gives the relative loss in the
           average run-length compared to the optimal test, i.e., $\tilde{E}_0[\tau_{\text{opt}}]$/$\tilde{E}_0[\tau_{\text{Wald}}]$}
  \label{tb:iid_wald}
\end{table}

From Table \ref{tb:iid_opt} and \ref{tb:iid_wald} it can be seen that the optimal sequential test can perform significantly better than the test using
Wald's approximations. In particular, in cases where large overshoots over the threshold can be expected, i.e., for large error probabilities, the
average run-length is reduced by up to 25\%. For smaller error probabilities the improvement is less pronounced, as was expected.

How the optimal thresholds can be obtained from the results of the LP problem is illustrated in Figure~\ref{fig:iid_cost}. Here the costs for stopping
and continuing the test are plotted as functions of the likelihood ratio. The points of intersection correspond to the thresholds. It is noteworthy
that even if the target error probabilities are chosen to be identical, the values of the optimal cost coefficients differ significantly. This
difference can be explained as follows: Since the run-length is minimized under the null hypothesis, the likelihood ratio sequence admits a permanent
drift towards the lower threshold. Choosing the latter closer to zero significantly reduces the run-length at the cost of an increased probability of
second type errors. The probability of first type errors, by contrast, is mainly determined by the upper threshold, which has very little influence on
the run-length under $P_0$. Consequently, first type errors have to be penalized much higher than second type errors if both are supposed to occur
with the same probability. This asymmetry highlights problems with approaches that assume the cost coefficients to be given a priori or simply assume
both error types to be equally costly.

\begin{figure}[!t]
  \centering
  \includegraphics{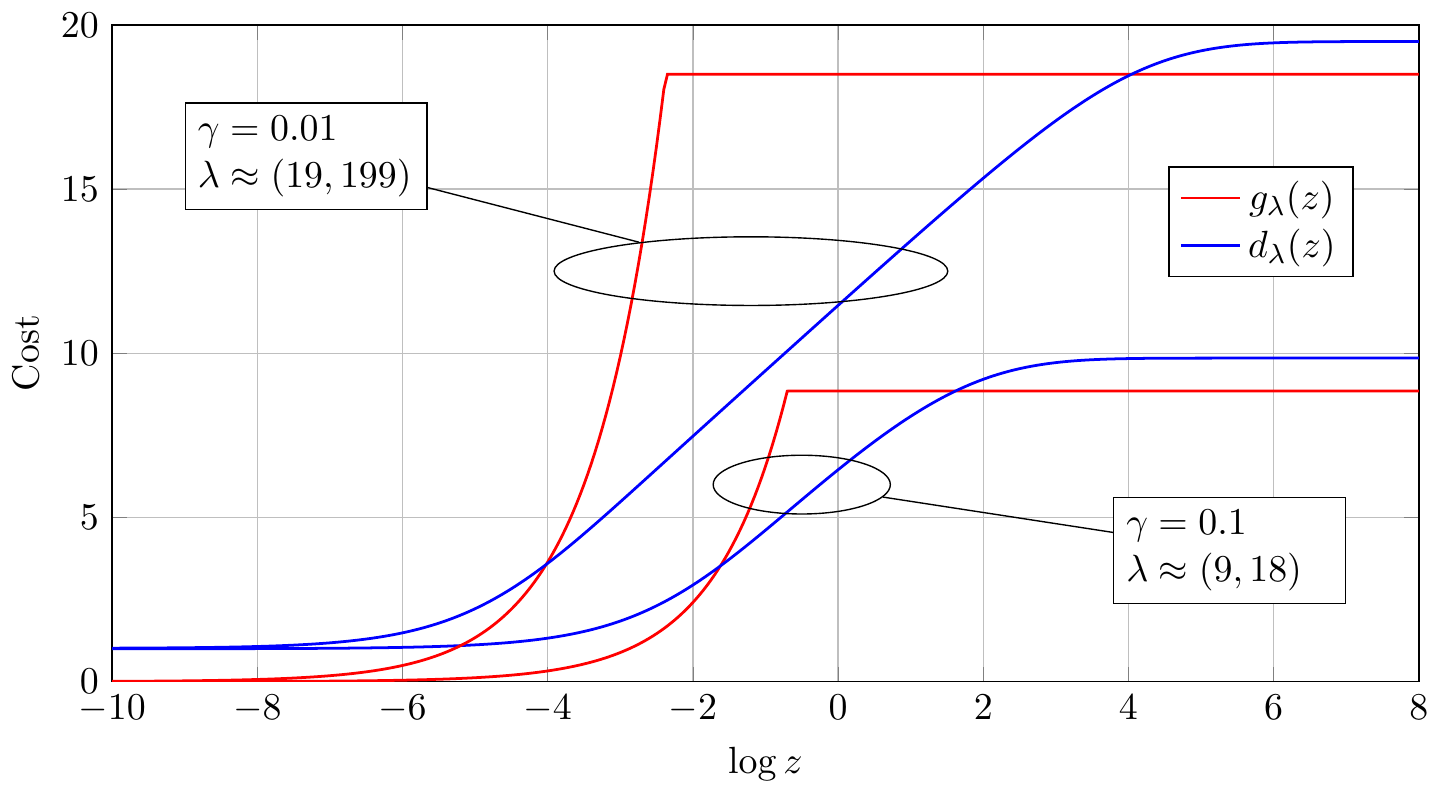}
  \caption{Mean Shifted Gaussian IID: Cost functions for optimal tests with error probabilities $\gamma = 0.1$ and $\gamma = 0.01$.}
  \label{fig:iid_cost}
\end{figure}

%----------------------------------------------------------------------------------------------------------------------------------------------------%
\subsection{Observable Markov Chain}
%----------------------------------------------------------------------------------------------------------------------------------------------------%
The above example can be complicated by assuming that the observed random sequence is governed by an observable Markov chain with
state space $E_{\theta} = \{1,2\}$. More precisely,
\begin{align*}
  \mathcal{H}_0: & \quad X_n := (Y_n,\Theta_n) \sim p_0(\theta_n) \cdot f_{\mathcal{N}}(y_n \,;\, 0, \sigma) \\
  \mathcal{H}_1: & \quad X_n := (Y_n,\Theta_n) \sim p_1(\theta_n | \theta_{n-1}) \cdot f_{\mathcal{N}}(y_n \,;\, \theta_n/2, \sigma),
\end{align*}
where $p_i$ defines the (transition) probabilities of the states. In this model, $Y_n$ and $\Theta_n$ are independent Gaussian and Bernoulli random
variables under the null hypothesis, while under the alternative hypothesis the distribution of $Y_n$ depends on the current state $\theta_n$. The
initial state is assumed to be $\theta_0 = 1$. Apparently, $\theta_{n-1}$ is a sufficient statistic for the distribution of $X_n$, conditioned on the
previous observations. The integral kernel under the null hypothesis can be shown to be a scaled and shifted version of \eqref{eq:iid_kernel}, namely
\begin{equation*}
  h^0_{\text{MC}}(s',\theta',s,\theta \,;\, \sigma) = p_0(\theta') \cdot f_{\mathcal{N}}(s'-s \,;\, \mu_{\text{MC}},\sigma_{\text{MC}}),
\end{equation*}
where
\begin{equation*}
  \mu_{\text{MC}} = \mu_{\text{MC}}(\theta',\theta) = -\frac{1}{8}\frac{\theta'^2}{\sigma^2} + \log \frac{p_1(\theta'|\theta)}{p_0(\theta')} \quad
\text{and} \quad \sigma_{\text{MC}} = \sigma_{\text{MC}}(\theta') = \frac{1}{2} \frac{\theta'}{\sigma}.
\end{equation*}

For the numerical results $\sigma = 1$ is assumed and the transition probabilities under $\mathcal{H}_1$ are chosen symmetrically as $p(\theta' \,|\,
\theta) = 0.8$ for $\theta' = \theta$ and $p(\theta' \,|\, \theta) = 0.2$ for $\theta' \neq \theta$. Under $\mathcal{H}_0$ $p_0(1) = p_0(2) =
0.5$ is used. Again, $t_{0.5}(z)$ was sampled at 200 points. However, since $\rho_{\lambda}$ is now defined on $\mathbb{R}_+ \times \{1,2\}$, the
stacked vector $\boldsymbol{\rho} = (\boldsymbol{\rho}(\boldsymbol{z},1), \boldsymbol{\rho}(\boldsymbol{z},2))$ is of size $400$ and the matrix
$\boldsymbol{H}^0_{\text{MC}}$ of size $400 \times 400$. The runtime of the solver is not significantly affected by this rise in complexity.

\begin{table}[!t]
  \centering
  \includegraphics{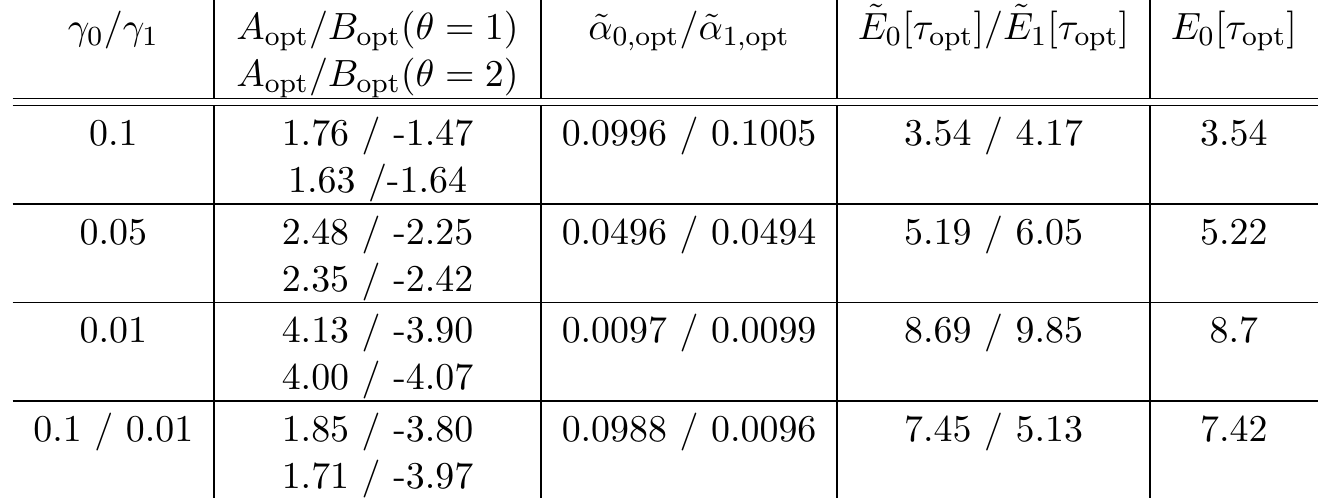}
  \caption{Observable Markov Chain: Optimal log-likelihood ratio thresholds $A(\theta)$, $B(\theta)$, empirical error probabilities
           $\tilde{\alpha}_i$, and average and expected run-length $\tau$ for target error probabilities $\gamma$}
  \label{tb:mc_opt}
  \includegraphics{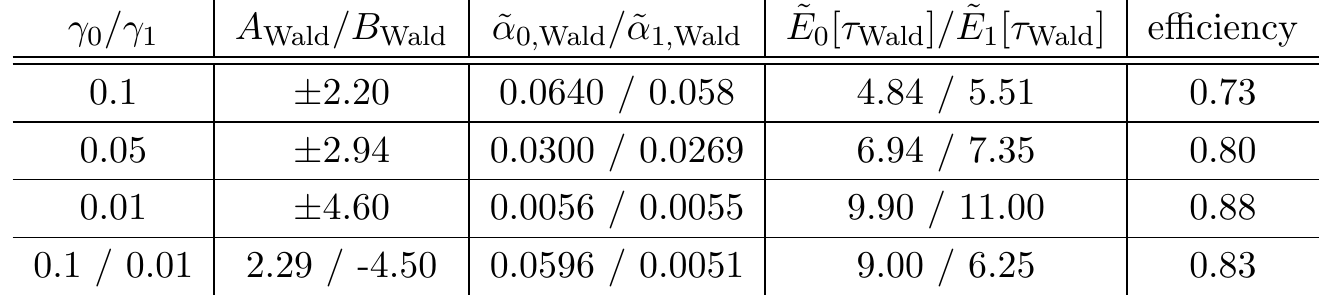}
  \caption{Observable Markov Chain: Log-likelihood ratio thresholds $A$, $B$, empirical error probabilities $\tilde{\alpha}_i$, and
           average run-length $\tau$ for target error probabilities $\gamma$ using Wald's approximations. The last column gives the relative loss in
           the average run-length compared to the optimal test, i.e., $\tilde{E}_0[\tau_{\text{opt}}]$/$\tilde{E}_0[\tau_{\text{Wald}}]$}
  \label{tb:mc_wald}
\end{table}

The results of the optimal test are given in Table \ref{tb:mc_opt} and Figure \ref{fig:mc_cost}. The expected run-length and error probabilities of a
test using Wald's approximations are shown in Table \ref{tb:mc_wald}. The results do not differ much from the i.i.d.\ scenario in terms of the
efficiency of Wald's test. The reduction in samples by using the optimal strategy is still between 25\% and 10\%. However, to achieve this reduction,
the likelihood ratio alone is no longer a sufficient test statistic since different thresholds have to be used in different states---see Figure
\ref{fig:mc_cost}. In line with the asymptotic optimality of Wald's approximations, the difference between the thresholds in the two states reduces
with decreasing error probabilities.

\begin{figure}[!t]
  \centering
  \includegraphics{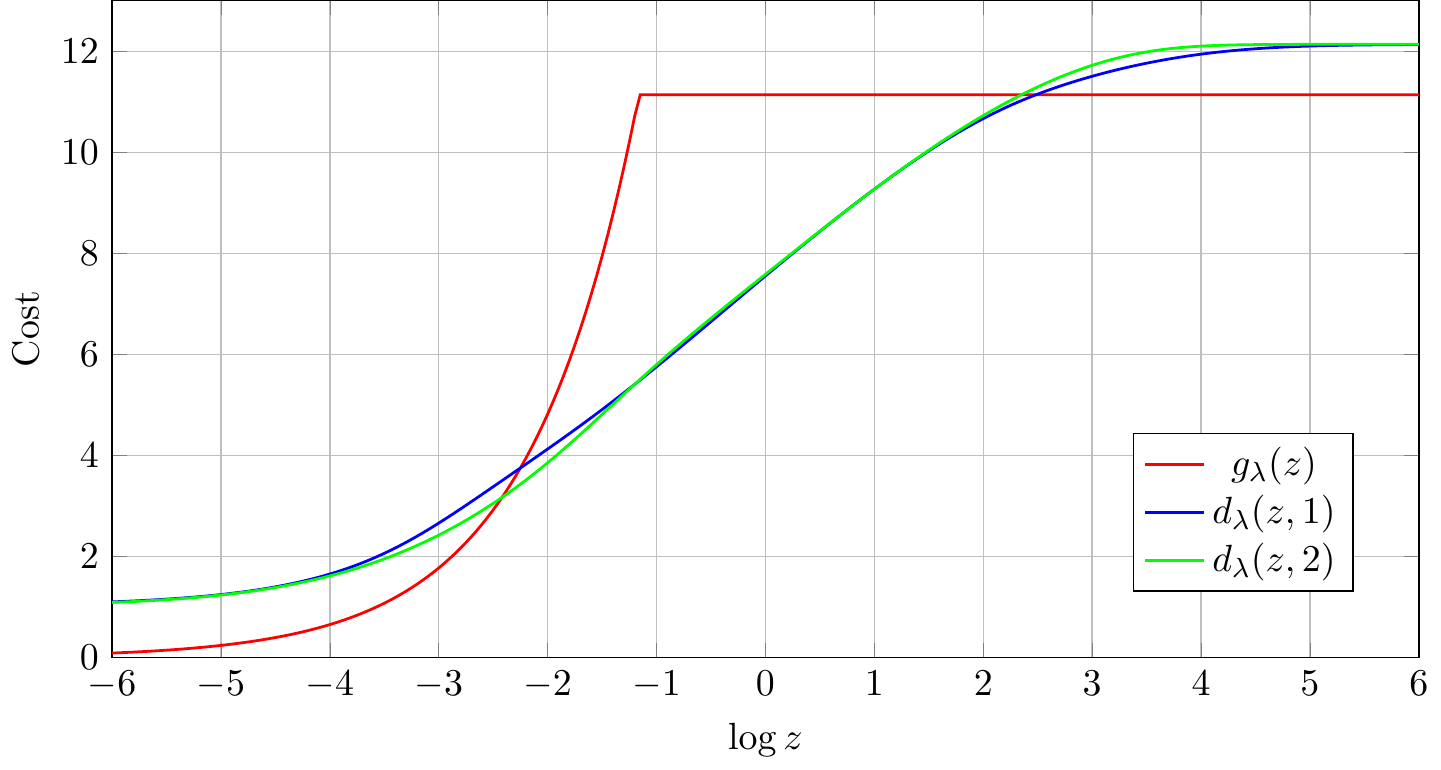}
  \caption{Observable Markov Chain: State dependent cost functions for an optimal test with error probabilities $\gamma = 0.05$.}
  \label{fig:mc_cost}
\end{figure}

%----------------------------------------------------------------------------------------------------------------------------------------------------%
\subsection{Gaussian AR(1) Process}
%----------------------------------------------------------------------------------------------------------------------------------------------------%
The final example is the Gaussian AR(1) process. In \cite{Novikov2009} it is shown that the optimal stopping strategy for this process is a
function of the likelihood ratio and the current observation. However, to the best of our knowledge, the exact strategy has never been derived, let
alone implemented.

The two hypotheses are given by
\begin{align*}
  \mathcal{H}_0: & \quad X_n = a_0 X_{n-1} + \epsilon_n \\
  \mathcal{H}_1: & \quad X_n = a_1 X_{n-1} + \epsilon_n,
\end{align*}
where $(\epsilon_n)_{n\geq1}$ is a sequence of i.i.d.\ zero mean Gaussian random variables with standard deviation $\sigma$. Since knowledge of
$x_{n-1}$ is sufficient to describe the conditional distribution of $X_n$, $\theta_{n-1} = x_{n-1}$ is chosen. The log-likelihood ratio increment
$\Delta s_n$ of the single observation $x_n$ is given by
\begin{equation*}
  \Delta s_n = \frac{(a_1-a_0)x_{n-1}}{\sigma^2} x_n - \frac{(a_1^2-a_0^2) x_{n-1}^2}{2\sigma^2}.
\end{equation*}
The joint distribution of $(S_n,\Theta_n)$, given $(s_{n-1},\theta_{n-1})$, is nonzero only on the one-dimensional manifold
\begin{equation}
  x_n = \frac{\sigma^2}{x_{n-1}}\frac{s_n -s_{n-1}}{a_1-a_0} + \frac{x_{n-1}}{2} (a_1+a_0).
  \label{eq:AR1_manifold}
\end{equation}
Let the set of $x_n$ that satisfy \eqref{eq:AR1_manifold} be denoted $\mathcal{X}(s_n,s_{n-1},x_{n-1}) \subset \mathbb{R}$. Using this notation, the
integral kernel under $\mathcal{H}_i$, $i=0,1$ is
\begin{equation*}
  h^i_{\text{AR1}}(s',\theta',s,\theta \,;\, a_0, a_1, \sigma) = \begin{cases}
                                                                    f_{\mathcal{N}}(\theta' \,;\, a_i \theta,\sigma), & \theta' \in
                                                                    \mathcal{X}(s',s,\theta) \\
                                                                    0, & \text{otherwise}.
                                                                 \end{cases}
\end{equation*}
For the experiment, $\sigma = 1$, $a_0 = 0$ and $a_1 = 1$ are chosen, which corresponds to testing between an AR(1) process and Gaussian noise. In
contrast to the previous examples, there now are two continuous quantities to discretize. Again, $200$ sampling points are used for each, i.e., the
log-likelihood ratio and the current observation $x_n$. Consequently, the vector $\boldsymbol{\rho}$ is of size $4\cdot10^4$ and
$\boldsymbol{H}_{\text{AR1}}$ of size $4\cdot10^4 \times 4\cdot10^4$. Problems of this size can still be handled by state of the art hard- and
software, especially since $\boldsymbol{H}_{\text{AR1}}$ is exceedingly sparse, but the limitations of the proposed method start to show. For more
complex dependency structures, more advanced solution methods have to be used.

\begin{table}[!t]
  \centering
  \includegraphics{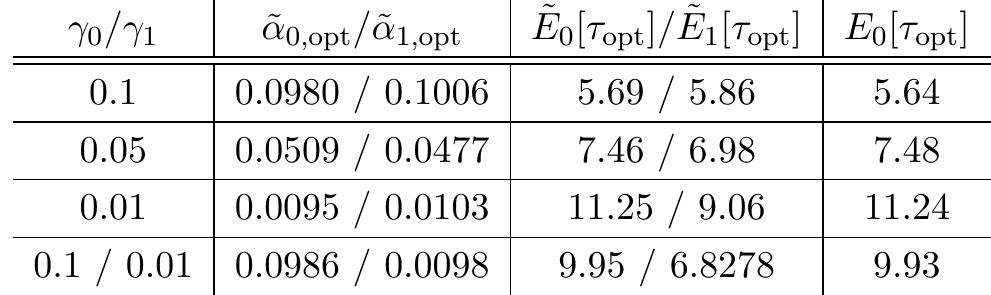}
  \caption{Gaussian AR(1) process: Empirical error probabilities $\tilde{\alpha}_i$ and average and expected run-length $\tau$ for target error
           probabilities $\gamma$}
  \label{tb:ar1_opt}
  \includegraphics{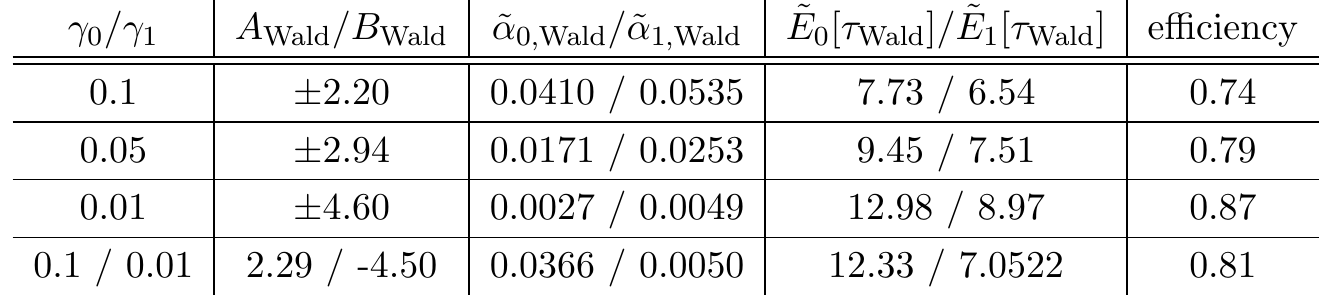}
  \caption{Gaussian AR(1) process: Log-likelihood ratio thresholds $A$, $B$, empirical error probabilities $\tilde{\alpha}_i$, and average
           run-length $\tau$ for target error probabilities $\gamma$ using Wald's approximations. The last column gives the relative loss
           in the average run-length compared to the optimal test, i.e., $\tilde{E}_0[\tau_{\text{opt}}]$/$\tilde{E}_0[\tau_{\text{Wald}}]$}
  \label{tb:ar1_wald}
\end{table}

The average run-length and the error probabilities of the optimal test and the one using Wald's approximations are given in Tables \ref{tb:ar1_opt}
and \ref{tb:ar1_wald}. A segment of the cost functions for $\gamma = 0.05$ is depicted in Figure \ref{fig:ar1_cost}. The intersection of the two
surfaces corresponds to the thresholds of the test. In Figure \ref{fig:ar1_thresholds}, the latter is shown together with the approximated constant
ones. Interestingly, the optimal thresholds are not uniformly tighter than the approximations. Instead, the additional degree of freedom is used to
loosen the thresholds for observations that are very unlikely under $P_0$ and tighten them in the critical region around the origin. Evidently, this
strategy is more efficient than uniformly tightening the thresholds. Another noteworthy fact is that in contrast to the lower threshold, the upper
threshold is far from being constant. This does not contradict the asymptotic optimality of the constant threshold test. It does, however,
indicate that there is no longer a stopping strategy that concurrently minimizes the run-lengths under both hypotheses, as is the case for i.i.d.\
observations \citep{Wald1948, Siegmund1985}. Minimizing the run-length under $\mathcal{H}_1$ yields a mirrored version of the thresholds in Figure
\ref{fig:ar1_thresholds}, with the lower threshold following the parabolic shape and vice versa.

A nice property of the optimal thresholds shown here is that they are relatively easy to approximate by polynomials or rational functions. In
practice, a few coefficients can therefore be sufficient to implement a nearly optimal strategy that combines the ease of the constant threshold test
with the efficiency of the optimal one.

\begin{figure}[!t]
  \centering
  \includegraphics{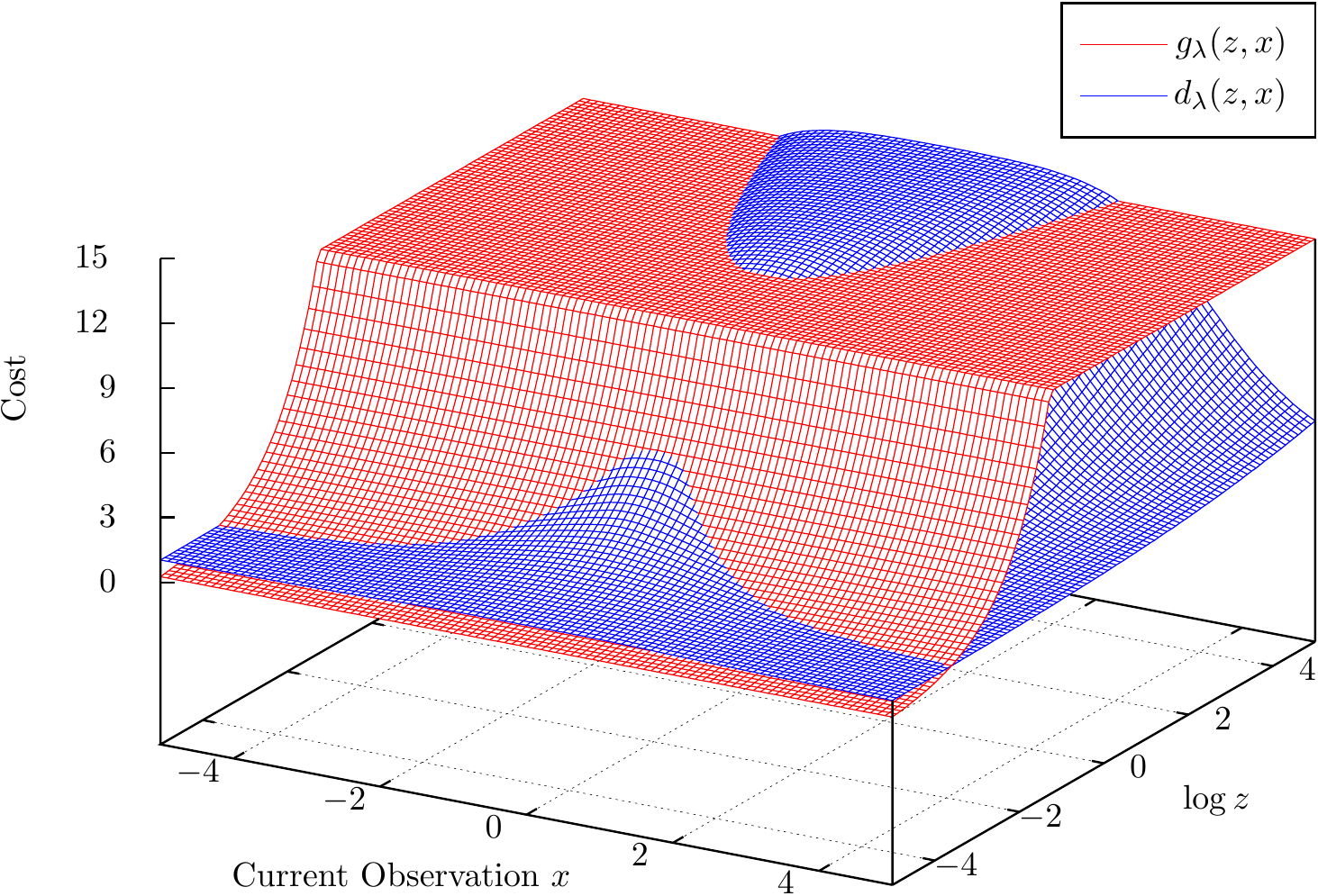}
  \caption{Gaussian AR(1) process: Segment of the cost functions for an optimal test with error probabilities $\gamma = 0.05$.}
  \label{fig:ar1_cost}
\end{figure}

\begin{figure}[!t]
  \centering
  \includegraphics{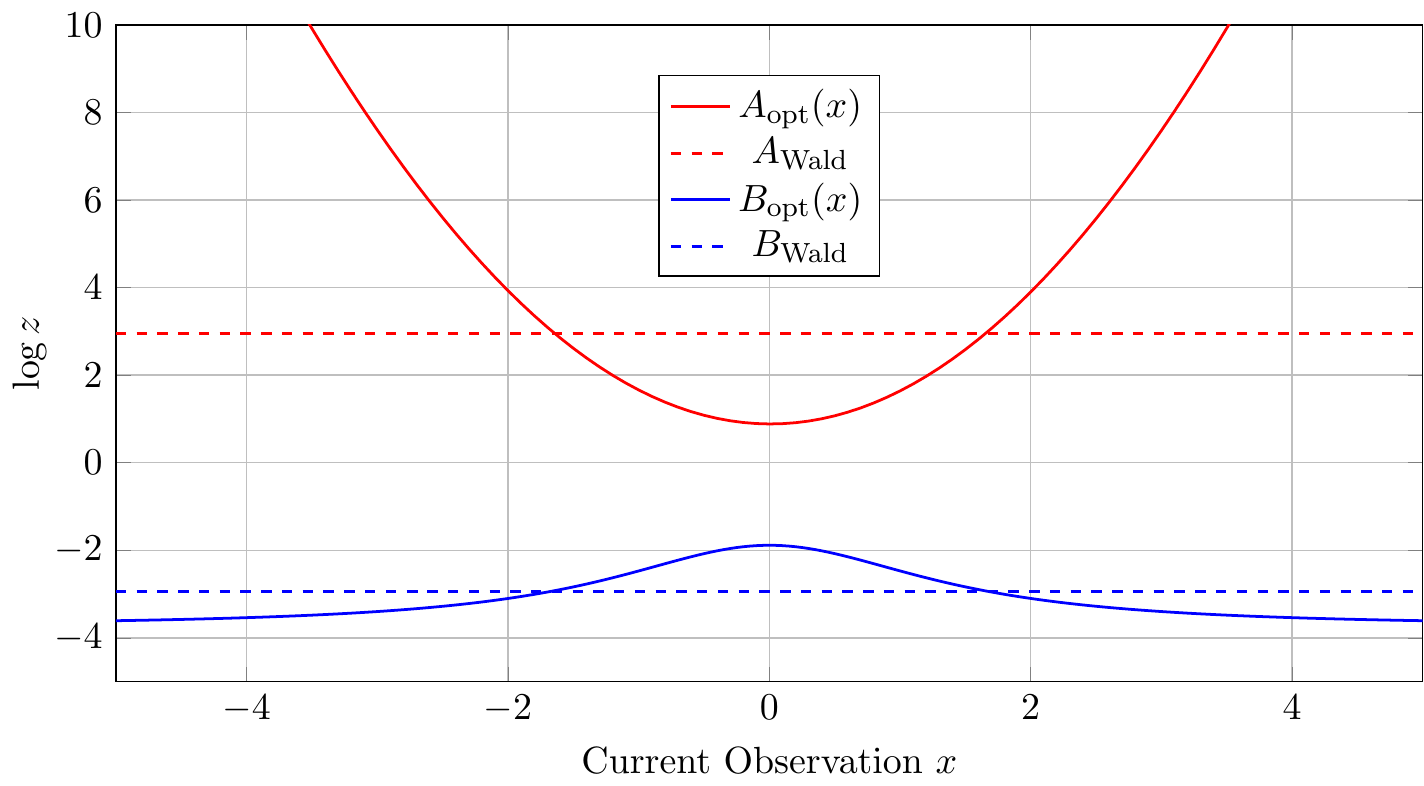}
  \caption{Gaussian AR(1) process: Optimal and approximated log-likelihood ratio thresholds as functions of the current observation $x_n$ for target
           error probabilities $\gamma = 0.05$.}
  \label{fig:ar1_thresholds}
\end{figure}

In addition to the results given in Table \ref{tb:ar1_opt}, an optimal test for the AR(1) model was designed with $\gamma = (0.0410, 0.0535)$, which
are the (empirical) error probabilities of the Wald test with target error probabilities $\gamma = 0.1$. The idea is to compare the strictly  optimal
test to the optimal constant threshold test. The expected run-length of the optimal test is 7.45, compared to 7.73 for the test with
constant thresholds. This corresponds to a reduction of about 3.6\%. Whether this improvement is worth the increased complexity surely depends on the
actual application. However, calculating the optimal constant thresholds is a non-trivial problem in itself so that the effort might as well be
invested in solving the problem exactly.

%%%%%%%%%%%%%%%%%%%%%%%%%%%%%%%%%%%%%%%%%%%%%%%%%%%%%%%%%%%%%%%%%%%%%%%%%%%%%%%%%%%%%%%%%%%%%%%%%%%%%%%%%%%%%%%%%%%%%%%%%%%%%%%%%%%%%%%%%%%%%%%%%%%%%%
\appendix
%%%%%%%%%%%%%%%%%%%%%%%%%%%%%%%%%%%%%%%%%%%%%%%%%%%%%%%%%%%%%%%%%%%%%%%%%%%%%%%%%%%%%%%%%%%%%%%%%%%%%%%%%%%%%%%%%%%%%%%%%%%%%%%%%%%%%%%%%%%%%%%%%%%%%%
\section{Proof of Theorem \ref{th:optimal_stopping}}
\label{apd:proof_optimal_stopping}
%%%%%%%%%%%%%%%%%%%%%%%%%%%%%%%%%%%%%%%%%%%%%%%%%%%%%%%%%%%%%%%%%%%%%%%%%%%%%%%%%%%%%%%%%%%%%%%%%%%%%%%%%%%%%%%%%%%%%%%%%%%%%%%%%%%%%%%%%%%%%%%%%%%%%%
Theorem \ref{th:optimal_stopping} is a corollary of fundamental results in optimal stopping theory and variations of it have been proved repeatedly in
the literature \citep{Shiryaev1978,Novikov2009}. The proof presented here does not differ significantly, but is included for the sake of completeness
and to introduce concepts and notations that are used throughout the paper.

Under the usual assumption that decisions are not allowed to depend on future observations, the minimal cost $V^*$ of a stopping problem is given by
the limit
\begin{equation*}
  V^* = \lim_{N \to \infty} V_{0,N},
\end{equation*}
where for each $N \geq 1$ the sequence $(V_{n,N})_{n \geq 0}$ is recursively defined by
\begin{equation*}
  V_{n,N} = \min \left\{ c_{n,N} \,,\, E[V_{n+1,N}|\mathcal{F}_n] \right\}
\end{equation*}
with basis $V_{N,N} = c_{N,N}$---see \citet{Shiryaev1978, Peskir2006, Poor2009}. Here $N$ denotes the finite time horizon of the truncated stopping
rule\footnote{A stopping rule  is said to be truncated, if it is guaranteed to stop after at most $N$ time instances, i.e., if an integer $N \geq
1$ exists such that $P(\tau \leq N) = 1$.}, $c_{n,N}$ the cost for stopping at time $n$ given horizon $N$, and $V_{n,N}$ the cost for stopping at
the optimal time instant between $n$ and $N$.

For the sequential detection problem, $c_{n,N}$ is obtained from \eqref{eq:stop_prob} and is given by
\begin{equation*}
  c_{n,N} = n + g_{\lambda}(z^n)
\end{equation*}
for all $N \geq 1$. Assuming that the optimal cost function is of the form $V_{n,N} = n + \rho_{\lambda}^{n,N}(z^n,\theta_n)$ for some $n < N$,it can
be shown, via induction, that
\begin{align*}
  V_{n-1,N} & = \min \left\{ (n-1) + g_{\lambda}(z^{n-1}) \,,\, E[V_{n,N}|\mathcal{F}_{n-1}] \right\} \\
          & = \min \left\{ (n-1) + g_{\lambda}(z^{n-1}) \,,\, n + E[\rho_{\lambda}^{n,N}(z^{n},\theta_{n})|\mathcal{F}_{n-1}] \right\} \\
          & = (n-1) + \min \left\{ g_{\lambda}(z^{n-1}) \,,\, 1 + E[\rho_{\lambda}^{n,N}(z^{n},\theta_{n})|\mathcal{F}_{n-1}] \right\} \\
          & = (n-1) + \rho_{\lambda}^{n-1,N}(z^{n-1},\theta_{n-1}),
\end{align*}
where
\begin{equation*}
  E[\rho_{\lambda}^{n,N}(z^{n},\theta_{n})|\mathcal{F}_{n-1}] = \int \rho_{\lambda}^{n,N} \left(
    z_0^{n-1}\frac{f_{0,\theta_{n-1}}}{f_{\theta_{n-1}}}, z_1^{n-1} \frac{f_{1,\theta_{n-1}}}{f_{\theta_{n-1}}}, \xi_{\theta_{n-1}} \right) \,
    \dint F_{\theta_{n-1}}
\end{equation*}
is a function of $z^{n-1}$ and $\theta_{n-1}$. The induction basis is given by $V_{N,N} = c_{N,N} = N + g_{\lambda}(z^N)$ so that
$\rho_{\lambda}^{N,N} =  g_{\lambda}$. Hence, the minimum cost of the $N$-truncated test is $V_{0,N} = \rho_{\lambda}^{0,N}(z^0,\theta_0)
= \rho_{\lambda}^{0,N}(1,1,\theta_0)$ for all $N \geq 1$.

To determine the limit $\rho^n_{\lambda} = \lim_{N \to \infty} \rho_{\lambda}^{n,N}$, note that $\rho_{\lambda}^{n,N}$ is obtained from
$\rho_{\lambda}^{n+1,N}$ by applying the transformation
\begin{equation}
  T\left\{ \rho(z,\theta) \right\} = \min \left\{ g_{\lambda} \; , \; 1 + \int \rho \left(z_0 \frac{f_{0,\theta}}{f_{\theta}}, z_1
    \frac{f_{1,\theta}}{f_{\theta}}, \xi_{\theta} \right) \, \dint F_{\theta} \right\},
  \label{eq:def_T}
\end{equation}
which is monotonic in $\rho$. Since $\rho_{\lambda}^{N,N} = g_{\lambda}$, independent of $N$, $\rho_{\lambda}^{n,N}$ can be expressed as
\begin{align*}
  \rho_{\lambda}^{n,N} = T^{N-n}\left\{g_{\lambda}\right\},
\end{align*}
where $T^n$ denotes an $n$-times repeated application of $T$. The transition to the non-truncated case yields
\begin{align*}
  \lim_{N \to \infty} \rho_{\lambda}^{n,N} = \lim_{N \to \infty} T^{N-n}\left\{g_{\lambda}\right\} = \lim_{N \to
  \infty} T^N \left\{g_{\lambda}\right\} =: \rho_{\lambda}
\end{align*}
for all $n \in \mathbb{N}$. To show that $\lim_{N \to \infty} T^N \left\{g_{\lambda}\right\}$ exists and is unique, it suffices to show that $T^n
\left\{g_{\lambda}\right\} \geq 0$ for all $n \geq 1$ and that the sequence $T^n \left\{g_{\lambda}\right\}$ is monotonically nonincreasing
\citep{Rudin1987, Novikov2009}. The fact that $T^n \left\{g_{\lambda}\right\} \geq 0$  follows directly from $g_{\lambda} \geq 0$ and the definition
of $T$. The monotonic property can again be established by induction. Assume that $T^n \left\{g_{\lambda}\right\} \leq T^{n-1}
\left\{g_{\lambda}\right\}$. By monotonicity of $T$ it then holds that
\begin{align*}
  T^n \left\{g_{\lambda}\right\} = T \left\{T^{n-1}\{g_{\lambda}\}\right\} \leq T \left\{T^n\{g_{\lambda}\}\right\} = T^{n+1}
\left\{g_{\lambda}\right\}.
\end{align*}
The induction basis $T\left\{g_{\lambda}\right\} \leq g_{\lambda}$ holds trivially since $T\left\{\rho\right\} \leq g_{\lambda}$ by definition. The
fixed-point solution of
\begin{align*}
  \rho = T\{\rho\}
\end{align*}
yields \eqref{eq:stop_prob}. This concludes the proof.

%%%%%%%%%%%%%%%%%%%%%%%%%%%%%%%%%%%%%%%%%%%%%%%%%%%%%%%%%%%%%%%%%%%%%%%%%%%%%%%%%%%%%%%%%%%%%%%%%%%%%%%%%%%%%%%%%%%%%%%%%%%%%%%%%%%%%%%%%%%%%%%%%%%%%%
\section{Proof of Lemma \ref{lm:uniform_convergence}}
\label{apd:proof_uniform_convergence}
%%%%%%%%%%%%%%%%%%%%%%%%%%%%%%%%%%%%%%%%%%%%%%%%%%%%%%%%%%%%%%%%%%%%%%%%%%%%%%%%%%%%%%%%%%%%%%%%%%%%%%%%%%%%%%%%%%%%%%%%%%%%%%%%%%%%%%%%%%%%%%%%%%%%%%
Uniform convergence of monotonic sequences often follows immediately from Dini's theorem \citep{Rudin1987}. In this case, however, neither is the
state space $E$ compact, nor is $\rho_{\lambda}$ necessarily continuous in $\theta$. Nevertheless, uniform convergence can still be shown via a detour
over almost uniform convergence.

Define the measure
\begin{equation*}
  H^*(B) = \sup_{(z,\theta) \in E} H_{z,\theta}(B), \quad B \in \mathcal{E}.
\end{equation*}
By Theorem \ref{th:optimal_stopping}, $(\rho_{\lambda}^n)_{n\geq0}$ converges pointwise on $E$ and hence $H^*$ almost everywhere. Egorov's theorem
\citep{Beals2010} states that this implies almost uniform convergence with respect to $H^*$, i.e., for every $\varepsilon > 0$, there exists a set
$B_{\varepsilon} \in \mathcal{E}$ such that $H^*(B_{\varepsilon}) < \varepsilon$ and $(\rho_{\lambda}^n)_{n\geq0}$ converges uniformly on $E \setminus
B_{\varepsilon}$. In the following it is shown that for $\rho_{\lambda}$ almost uniform convergence implies uniform convergence.

Since $(\rho_{\lambda}^n)_{n\geq0}$ is monotonically nonincreasing, $\rho_{\lambda}^n$ can be written as $\rho_{\lambda}^n = \rho_{\lambda} + \Delta
\rho_{\lambda}^n$ for every $n\geq0$, where $(\Delta \rho_{\lambda}^n)_{n\geq0}$ is a nonincreasing sequence of nonnegative functions. To guarantee
uniform convergence it suffices to show that $\lim_{n \to \infty} \sup_{(z,\theta) \in E} \Delta \rho_{\lambda}^n = 0$. By definition of
$\rho_{\lambda}^n$ it holds that
\begin{align*}
  \sup_{(z,\theta) \in E} \Delta \rho_{\lambda}^n & = \sup_{(z,\theta) \in E} \left\{ \min \left\{ g_{\lambda}(z) \;,\; 1 + \int
    \rho_{\lambda}^{n-1} \, \dint H_{z,\theta} \right\} - \rho_{\lambda}(z,\theta) \right\} \\
  & \leq \sup_{(z,\theta) \in E} \left\{ \min \left\{ g_{\lambda}(z) \;,\; 1 + \int \rho_{\lambda} \, \dint H_{z,\theta} \right\} - \rho_{\lambda} +
    \int \Delta \rho_{\lambda}^{n-1} \, \dint H_{z,\theta} \right\} \\
  & = \sup_{(z,\theta) \in E} \left\{ \int \Delta \rho_{\lambda}^{n-1} \, \dint H_{z,\theta} \right\}.
\end{align*}
With $B_{\varepsilon}$ defined as above, it further follows that
\begin{align*}
  \sup_{(z,\theta) \in E} \Delta \rho_{\lambda}^n & \leq \sup_{(z,\theta) \in E} \left\{ \int \Delta \rho_{\lambda}^{n-1} \, \dint H_{z,\theta}
    \right\} \\
  & \leq \int_{E \setminus B_{\varepsilon}} \; \sup_{(z,\theta) \in E \setminus B_{\varepsilon}} \Delta \rho_{\lambda}^{n-1} \, \dint H^* +
    \int_{B_{\varepsilon}} \; \sup_{(z,\theta) \in B_{\varepsilon}} \Delta \rho_{\lambda}^{n-1} \, \dint H^* \\
  & < \sup_{(z,\theta) \in E \setminus B_{\varepsilon}} \Delta \rho_{\lambda}^{n-1} + \varepsilon \sup_{(z,\theta) \in E} \Delta
    \rho_{\lambda}^{n-1}.
\end{align*}
Since $\lim_{n \to \infty} \sup_{(z,\theta) \in E \setminus B_{\varepsilon}} \Delta \rho_{\lambda}^n = 0$, the sequence $(\sup_{(z,\theta) \in E}
\Delta \rho_{\lambda}^n)_{n\geq0}$ converges to zero for every \mbox{$\varepsilon < 1$}, given that $\sup_{(z,\theta) \in E} \Delta \rho_{\lambda}^n$
is
bounded for some $n$. The latter is guaranteed by the pointwise convergence of $\rho_{\lambda}^n$ on $E$.

%%%%%%%%%%%%%%%%%%%%%%%%%%%%%%%%%%%%%%%%%%%%%%%%%%%%%%%%%%%%%%%%%%%%%%%%%%%%%%%%%%%%%%%%%%%%%%%%%%%%%%%%%%%%%%%%%%%%%%%%%%%%%%%%%%%%%%%%%%%%%%%%%%%%%%
\section{Proof of Theorem \ref{th:derivative}}
\label{apd:proof_derivative}
%%%%%%%%%%%%%%%%%%%%%%%%%%%%%%%%%%%%%%%%%%%%%%%%%%%%%%%%%%%%%%%%%%%%%%%%%%%%%%%%%%%%%%%%%%%%%%%%%%%%%%%%%%%%%%%%%%%%%%%%%%%%%%%%%%%%%%%%%%%%%%%%%%%%%%
The proof of Theorem \ref{th:derivative} is given in two parts. First, it is assumed that $\rho_{\lambda}$ is differentiable almost everywhere and, in
particular, allows for interchanging the order of integration and differentiation. In the second part, these assumptions are refined and justified.

Assuming differentiability, taking the derivative with respect to $\lambda_i$ on both sides of \eqref{eq:Wald_Bellman_modified} yields
\begin{align*}
  \rho'_{\lambda_i}(z,\theta) = \begin{cases}
                                  \displaystyle g'_{\lambda_i}(z), & \text{for } (z,\theta) \in \mathcal{S}_{\lambda} \\[0.5em]
                                  \displaystyle \frac{\partial}{\partial \lambda_i} \int \rho_{\lambda} \, \dint H_{z,\theta}, & \text{for }
                                    (z,\theta) \in \overline{\mathcal{S}}_{\lambda}.
                                \end{cases}
\end{align*}
On $\partial \mathcal{S}_{\lambda}$ the derivative is not defined. Assuming, for now, that
\begin{equation}
  \frac{\partial}{\partial \lambda_i} \int \rho_{\lambda} \, \dint H_{z,\theta} = \int \rho'_{\lambda_i} \, \dint H_{z,\theta}
  \label{eq:dev_int_change}
\end{equation}
it follows that
\begin{align}
  \rho'_{\lambda_i}(z,\theta) & = g'_{\lambda}(z) \boldsymbol{1}_{\mathcal{S}_{\lambda}} +  \left( \int \rho'_{\lambda_i} \, \dint H_{z,\theta}
    \right) \boldsymbol{1}_{\overline{\mathcal{S}}_{\lambda}} \notag \\
  & = g'_{\lambda}(z) \boldsymbol{1}_{\mathcal{S}_{\lambda}} + \left( \int_{\mathcal{S}_{\lambda}} g'_{\lambda_i} \, \dint H_{z,\theta}
      + \int_{\overline{\mathcal{S}}_{\lambda}} \rho'_{\lambda_i} \, \dint H_{z,\theta} \right)
      \boldsymbol{1}_{\overline{\mathcal{S}}_{\lambda}}. \label{eq:derivative_intermediate1}
\end{align}
On $\mathcal{S}_{\lambda}$ the derivative of $g_{\lambda}$ exists everywhere, except on $\{z_0 \lambda_0 = z_1 \lambda_1\}$, which is a null set
by Assumption 3, so that
\begin{align*}
  g'_{\lambda_i} = z_i \boldsymbol{1}_{\mathcal{S}_{\lambda_i}}.
\end{align*}
Substituted into \eqref{eq:derivative_intermediate1} yields
\begin{align}
  \rho'_{\lambda_i}(z,\theta) & = z_i \boldsymbol{1}_{\mathcal{S}_{\lambda_i}} + \left(\int_{ \left\{ \left(z_0 \frac{f_{0,\theta}}{f_{\theta}}, z_1
    \frac{f_{1,\theta}}{f_{\theta}}, \xi_{\theta} \right) \in \mathcal{S}_{\lambda_i} \right\} } z_i \frac{f_{i,\theta}}{f_{\theta}} \,
    \dint F_{\theta} + \int_{\overline{\mathcal{S}}_{\lambda}} \rho'_{\lambda_i} \, \dint H_{z,\theta} \right)
    \boldsymbol{1}_{\overline{\mathcal{S}}_{\lambda}} \notag \\
  & = z_i \boldsymbol{1}_{\mathcal{S}_{\lambda_i}} + \left( z_i \int_{ \left\{ \left(z_0 \frac{f_{0,\theta}}{f_{\theta}}, z_1
    \frac{f_{1,\theta}}{f_{\theta}}, \xi_{\theta} \right) \in \mathcal{S}_{\lambda_i} \right\} } \, \dint F_{i,\theta} +
    \int_{\overline{\mathcal{S}}_{\lambda}} \rho'_{\lambda_i} \, \dint H_{z,\theta} \right) \boldsymbol{1}_{\overline{\mathcal{S}}_{\lambda}} \notag
    \\
  & = z_i \boldsymbol{1}_{\mathcal{S}_{\lambda_i}} + \left( z_i H^i_{z,\theta}(\mathcal{S}_{\lambda_i}) +
    \int_{\overline{\mathcal{S}}_{\lambda}} \rho'_{\lambda_i} \, \dint H_{z,\theta} \right) \boldsymbol{1}_{\overline{\mathcal{S}}_{\lambda}},
    \label{eq:fredholm_int}
\end{align}
which is the statement in Theorem \ref{th:derivative}. \pagebreak

At this point, it still needs to be shown that $\rho_{\lambda}$ is differentiable on $\overline{\mathcal{S}}$, that the order of integration and
differentiation in \eqref{eq:dev_int_change} can indeed be interchanged and that the integrals in \eqref{eq:fredholm_int} have a unique solution.

First, the question of differentiation under the integral is addressed. By the derivative lemma \citep{Bauer2001}, \eqref{eq:dev_int_change}
holds, if
\begin{enumerate}
  \item $\rho_{\lambda}$ is $H_{z,\theta}$ integrable for all $(z,\theta) \in E$ (already shown)
  \item $\rho_{\lambda}$ is differentiable almost everywhere
  \item an $H_{z,\theta}$ integrable function $\overline{r}$ that is independent of $\lambda$ exists with $\overline{r} \geq \lvert
        \rho'_{\lambda_i} \rvert$.
\end{enumerate}
Properties 2 and 3 can again be shown by induction. Consider the three sequences of functions $(R_{\lambda}^n)_{n\geq0}$ and
$(r_{\lambda_i}^n)_{n\geq0}$, $i=0,1$, with $R_{\lambda}^n: E \to \mathbb{R}_+$ recursively defined by
\begin{equation*}
  R_{\lambda}^n = g_{\lambda}(z) \boldsymbol{1}_{\mathcal{S}_{\lambda}} +  \left( \int R_{\lambda}^{n-1} \, \dint H_{z,\theta} \right)
    \boldsymbol{1}_{\overline{\mathcal{S}}_{\lambda}}
\end{equation*}
and $r^n_{\lambda_i}:  \overline{\mathcal{S}}_{\lambda} \to \mathbb{R}_+$ by
\begin{equation*}
  r_{\lambda_i}^n(z,\theta) = z_i H^i_{z,\theta}(\mathcal{S}_{\lambda_i}) + \int_{\overline{\mathcal{S}}_{\lambda}} r_{\lambda_i}^{n-1} \,
  \dint H_{z,\theta}.
\end{equation*}
The induction bases are given by $R_{\lambda}^0 = g_{\lambda}$ and $r_{\lambda_i}^0 = z_i$, respectively. Following the same line of arguments as in
the proofs of Theorem \ref{th:optimal_stopping} and Lemma \ref{lm:uniform_convergence}, it can be shown that $(R_{\lambda}^n)_{n\geq0}$ converges
monotonically and uniformly to a  unique, nonnegative function $R_{\lambda}$. Analogously, $(r^n_{\lambda_i})_{n\geq0}$ converges monotonically and
uniformly to the nonnegative function  $r_{\lambda_i}$. From the uniqueness of $\rho_{\lambda}$ and $R_{\lambda}$, it further follows that
\mbox{$R_{\lambda}  = \rho_{\lambda}$}.

Now, assume that $R_{\lambda}^n$ fulfills the differentiation lemma and that $r_{\lambda_i}^n(z,\theta)$ is its derivative with respect to
$\lambda_i$ on $\overline{S}_{\lambda}$. It then holds that on $\overline{S}_{\lambda}$,
\begin{align*}
  \frac{\partial}{\partial \lambda_i} R_{\lambda}^{n+1} = \int \frac{\partial}{\partial \lambda_i} R_{\lambda_i}^n \, \dint H_{z,\theta}
  = z_i H^i_{z,\theta}(\mathcal{S}_{\lambda_i}) + \int_{\overline{\mathcal{S}}_{\lambda}} r_{\lambda_i}^n \, \dint H_{z,\theta} = r_{\lambda_i}^{n+1},
\end{align*}
meaning that $\frac{\partial}{\partial \lambda_i} R_{\lambda}^{n+1}$ is well defined on $\overline{S}_{\lambda}$ and is upper bounded  by
$\overline{r} = r_{\lambda_i}^0 = z_i$. Consequently, $R_{\lambda}^{n+1}$ again fulfills the differentiation lemma. The fact that on
$\overline{\mathcal{S}}$ $\frac{\partial}{\partial  \lambda_i} R_{\lambda}^0 = \frac{\partial}{\partial \lambda_i} g_{\lambda} = z_i =
r_{\lambda_i}^0$ completes the induction.

In summary: The sequence $(R_{\lambda}^n)_{n\geq0}$ converges uniformly to $\rho_{\lambda}$ on $E$. The sequences of derivatives
$(r_{\lambda}^n)_{n\geq0}$ converge uniformly to $r_{\lambda_i}$ on $\overline{\mathcal{S}}_{\lambda}$. From this it follows \citep{Rudin1987} that
$\rho_{\lambda}$ is differentiable on $\overline{\mathcal{S}}_{\lambda}$ with $\rho'_{\lambda_i} = r_{\lambda_i} \leq \overline{r}_i = z_i$, which
justifies the use of the derivative lemma in  \eqref{eq:dev_int_change}.

%%%%%%%%%%%%%%%%%%%%%%%%%%%%%%%%%%%%%%%%%%%%%%%%%%%%%%%%%%%%%%%%%%%%%%%%%%%%%%%%%%%%%%%%%%%%%%%%%%%%%%%%%%%%%%%%%%%%%%%%%%%%%%%%%%%%%%%%%%%%%%%%%%%%%%
\section{Proof of Theorem \ref{th:error_prob}}
\label{apd:proof_error_prob}
%%%%%%%%%%%%%%%%%%%%%%%%%%%%%%%%%%%%%%%%%%%%%%%%%%%%%%%%%%%%%%%%%%%%%%%%%%%%%%%%%%%%%%%%%%%%%%%%%%%%%%%%%%%%%%%%%%%%%%%%%%%%%%%%%%%%%%%%%%%%%%%%%%%%%%
The error probabilities $\alpha_i$ are given by
\begin{equation*}
  \alpha_i = P_i[\phi_{\tau} = 1-i] = P_i[\phi_{\tau} = 1-i \,|\, z^0 = 1, \theta^0 = \theta_0 ].
\end{equation*}
Since the underpinning Markov process is time homogeneous and the optimal stopping rule time invariant, it further holds that for all $n,m \geq 0$
\begin{equation*}
  P_i[\phi_{\tau} = 1-i \,|\, z^n = z, \theta^n = \theta, \tau \geq n ] = P_i[\phi_{\tau} = 1-i \,|\, z^m = z, \theta^m = \theta, \tau \geq m ].
\end{equation*}
It is therefore possible to define the function
\begin{equation*}
  \alpha_i(z,\theta) = P_i[\phi_{\tau} = 1-i \,|\, z^n = z, \theta^n = \theta, \tau \geq n ],
\end{equation*}
which is independent of $n$. From the Chapman-Kolmogorov backward equations \citep{Peskir2006} and the definition of $H^i_{z,\theta}$ in
\eqref{eq:definition_H} it follows that
\begin{equation*}
  \alpha_i(z,\theta) = \int \alpha_i \, \dint H^i_{z,\theta}.
\end{equation*}
For $\psi = \psi_{\lambda}^*$ the stopping region is given by $\mathcal{S}_{\lambda}$. Choosing $\phi = \phi_{\lambda}^*$ further yields
\begin{equation*}
  \alpha_i^n(\mathcal{S}_{\lambda_i}) = 1 \quad \text{and} \quad \alpha_i^n(\mathcal{S}_{\lambda_{1-i}}) = 0
\end{equation*}
for all $n\geq0$. Hence,
\begin{equation}
  \alpha_i(z,\theta) = H^i_{z,\theta}(\mathcal{S}_{\lambda_i}) + \int_{\overline{\mathcal{S}}_{\lambda}} \alpha_i \, \dint H^i_{z,\theta}
  \label{eq:fredholm_int_errors}
\end{equation}
on $\overline{\mathcal{S}}_{\lambda}$. These are the well established Fredholm integral equations describing the error probabilities of
time homogeneous sequential tests \citep{Feller1971,Tartakovsky2014}. Note the difference in the integration measure compared to
\eqref{eq:derivative_lemma_eq}. Using the same techniques as before, it can be shown that \eqref{eq:fredholm_int_errors} has unique, positive
solutions on $\overline{\mathcal{S}}_{\lambda}$. What is left to show is that these solutions coincide with $\rho_{\lambda_i}'/z_i$.

From Theorem \ref{th:derivative} it is clear that on $\mathcal{S}_{\lambda_i}$
\begin{equation*}
  \frac{\rho_{\lambda_i}'}{z_i} = \frac{z_i}{z_i} = 1 = \alpha_i
\end{equation*}
and on $\mathcal{S}_{\lambda_{1-i}}$
\begin{equation*}
  \frac{\rho_{\lambda_i}'}{z_i} = 0 = \alpha_i.
\end{equation*}
On $\overline{\mathcal{S}}_{\lambda}$ it needs to be shown that $\rho_{\lambda_i}'/z_i$ solves \eqref{eq:fredholm_int_errors}, i.e.,
\begin{align*}
  \frac{\rho_{\lambda_i}'(z,\theta)}{z_i} & = H^i_{z,\theta}(\mathcal{S}_{\lambda_i}) + \int_{\overline{\mathcal{S}}_{\lambda}}
    \frac{\rho_{\lambda_i}'\left(z_0 \frac{f_{0,\theta}}{f_{\theta}}, z_1 \frac{f_{1,\theta}}{f_{\theta}}, \xi_{\theta} \right)}{z_i
    \frac{f_{i,\theta}}{f_{\theta}}} \, \dint F^i_{\theta} \\
  \frac{\rho_{\lambda_i}'(z,\theta)}{z_i} & = H^i_{z,\theta}(\mathcal{S}_{\lambda_i}) + \frac{1}{z_i} \int_{\overline{\mathcal{S}}_{\lambda}}
    \rho_{\lambda_i}'\left(z_0 \frac{f_{0,\theta}}{f_{\theta}}, z_1 \frac{f_{1,\theta}}{f_{\theta}}, \xi_{\theta} \right) \, \dint F_{\theta} \\
  \rho_{\lambda_i}'(z,\theta) & = z_i H^i_{z,\theta}(\mathcal{S}_{\lambda_i}) + \int_{\overline{\mathcal{S}}_{\lambda}}
    \rho_{\lambda_i}' \, \dint H_{z,\theta},
\end{align*}
which is true by Theorem \ref{th:derivative}.

%%%%%%%%%%%%%%%%%%%%%%%%%%%%%%%%%%%%%%%%%%%%%%%%%%%%%%%%%%%%%%%%%%%%%%%%%%%%%%%%%%%%%%%%%%%%%%%%%%%%%%%%%%%%%%%%%%%%%%%%%%%%%%%%%%%%%%%%%%%%%%%%%%%%%%
\section{Proof of Theorem \ref{th:no_gap}}
\label{apd:proof_no_gap}
%%%%%%%%%%%%%%%%%%%%%%%%%%%%%%%%%%%%%%%%%%%%%%%%%%%%%%%%%%%%%%%%%%%%%%%%%%%%%%%%%%%%%%%%%%%%%%%%%%%%%%%%%%%%%%%%%%%%%%%%%%%%%%%%%%%%%%%%%%%%%%%%%%%%%%
Assuming that $L_{\gamma}$ attains its maximum for some finite and positive $\lambda_{\gamma}^*$, and that its derivatives in this point exist, it
holds that
\begin{equation*}
  \frac{\partial}{\partial \lambda_i} L_{\gamma}(\lambda) \Big\vert_{\lambda = \lambda_{\gamma}^*} = 0, \quad i=0,1.
\end{equation*}
By Theorems \ref{th:derivative} and \ref{th:error_prob}
\begin{align}
  \frac{\partial}{\partial \lambda_i} L_{\gamma}(\lambda) & = \frac{\partial}{\partial \lambda_i} \rho_{\lambda}(1,1,\theta_0) -
      \frac{\partial}{\partial \lambda_i} (\lambda_0 \gamma_0 + \lambda_1 \gamma_1) \notag \\
  & = \rho'_{\lambda_i}(1,1,\theta_0) - \gamma_i \notag \\
  & = \alpha_i(\phi_{\lambda}^*,\psi_{\lambda}^*) - \gamma_i \label{eq:L_derivative}
\end{align}
so that for $\lambda = \lambda_{\gamma}^*$
\begin{equation*}
  \alpha_i(\phi_{\lambda_{\gamma}^*}^*,\psi_{\lambda_{\gamma}^*}^*) = \gamma_i.
\end{equation*}
The test therefore meets the target error probabilities exactly and is of minimum expected run-length by definition of $\rho_{\lambda}$. The dual
objective accordingly is
\begin{align*}
  L_{\gamma}(\lambda_{\gamma}^*) & = \rho_{\lambda_{\gamma}^*}(1,1,\theta_0) - \lambda_{\gamma,0}^* \gamma_0 - \lambda_{\gamma,1}^* \gamma_1 \\
  & = V_{\lambda_{\gamma}^*}(\psi_{\lambda_{\gamma}^*}^*) - \lambda_{\gamma,0}^* \gamma_0 - \lambda_{\gamma,1}^* \gamma_1 \\
  & = E[\tau(\psi_{\lambda_{\gamma}^*}^*)] + \lambda_{\gamma,0}^* \alpha_0(\phi_{\lambda_{\gamma}^*}^*,\psi_{\lambda^*}^*) + \lambda_{\gamma,1}^*
      \alpha_1(\phi_{\lambda_{\gamma}^*}^*,\psi_{\lambda_{\gamma}^*}^*) - \lambda_{\gamma,0}^* \gamma_0 - \lambda_{\gamma,1}^* \gamma_1 \\
  & = E[\tau(\psi_{\lambda_{\gamma}^*}^*)] = E[\tau(\psi_{\gamma}^*)].
\end{align*}
Finally, it needs to be shown that $\lambda_{\gamma}^*$ is indeed positive and finite. First, let some $\lambda_{\gamma,i}^* = 0$. This implies
$\rho_{\lambda} = 0$, which corresponds to a trivial test---see the remark at the end Section \ref{sec:properties}. Since in this case $\alpha_i = 1$,
it follows from
\eqref{eq:L_derivative} that
\begin{equation*}
  \frac{\partial}{\partial \lambda_{\gamma,i}^*} L_{\gamma}(\lambda_{\gamma,i}^*) = 1-\gamma_i > 0,
\end{equation*}
which contradicts the assumption that $\lambda_{\gamma}^*$ maximizes $L_{\gamma}$. Second, increasing some $\lambda_i$ indefinitely leads
to a test with $\alpha_i \to 0$ so that
\begin{equation*}
  \lim_{\lambda_i \to \infty} \frac{\partial}{\partial \lambda_i} L_{\gamma}(\lambda) = 0 - \gamma_i < 0.
\end{equation*}
Therefore, $L_{\gamma}(\lambda)$ is decreasing in the limit and, in turn, $\lambda_{\gamma}^*$ is bounded.

%%%%%%%%%%%%%%%%%%%%%%%%%%%%%%%%%%%%%%%%%%%%%%%%%%%%%%%%%%%%%%%%%%%%%%%%%%%%%%%%%%%%%%%%%%%%%%%%%%%%%%%%%%%%%%%%%%%%%%%%%%%%%%%%%%%%%%%%%%%%%%%%%%%%%%
\section{Enforcing Equality in the Constraint of the Relaxed Linear Program}
\label{apd:enforce_equality}
%%%%%%%%%%%%%%%%%%%%%%%%%%%%%%%%%%%%%%%%%%%%%%%%%%%%%%%%%%%%%%%%%%%%%%%%%%%%%%%%%%%%%%%%%%%%%%%%%%%%%%%%%%%%%%%%%%%%%%%%%%%%%%%%%%%%%%%%%%%%%%%%%%%%%%
If numerical problems arise in the solution of \eqref{eq:convex_problem_max}, such that the inequality constraint is not fulfilled with equality, a
regularization term can be added to the objective function, namely,
\begin{align}
   & \max_{\lambda > 0, \rho \in \mathcal{L}} \quad \rho(1,1,\theta_0) - \lambda_0 \gamma_0 - \lambda_1 \gamma_1 + c \int \rho \, \dint \eta
    \label{eq:augmented_problem} \\
  \text{s.t.} \quad & \rho(z,\theta) \leq \min \left\{ \lambda_0 z_0 \,,\, \lambda_1 z_1 \,,\, 1 + \int \rho \, \dint H_{z,\theta}
    \right\} \quad \forall (z,\theta) \in E \notag,
\end{align}
where $\eta$ is some strictly increasing measure on $(E,\mathcal{E})$ and $c$ is a small positive constant. In \eqref{eq:augmented_problem} $\rho$ is
explicitly maximized over the entire state space, whereas in \eqref{eq:convex_problem_max} this maximization resulted indirectly from maximizing
$\rho(1,1,\theta_0)$. Note that this regularization of the original problem is by no means the only way to combat numerical artifacts and is not
essential to this work. Nevertheless, it is straightforward and yields good results in practice.

Since $\rho$ is increasing in $\lambda$, $c$ has to be chosen small enough such that the problem is still bounded. To guarantee this, the additional
integral term can be upper bounded by
\begin{equation*}
  \int \rho \, \dint \eta < \int g \, \dint \eta < \lambda_0 \int z_0 \, \dint \eta + \lambda_0 \int z_1 \, \dint \eta = \lambda_0 + \lambda_1,
\end{equation*}
where, without loss of generality, it is assumed that $\eta$ has been chosen such that
\begin{equation*}
  \int z_i \, \dint \eta = 1, \quad i=0,1
\end{equation*}
The regularized objective function in \eqref{eq:augmented_problem} is then bounded by
\begin{equation*}
  \rho(1,1,\theta_0) - \lambda_0 (\gamma_0-c) - \lambda_1 (\gamma_1-c).
\end{equation*}
Consequently, choosing $c < \min\{\gamma_0,\gamma_1\}$ guarantees boundedness. Furthermore, this shows that the regularized problem corresponds to
the original problem with smaller target error probabilities, which means that the solution, even though not strictly optimal anymore, still satisfies
the original error requirements.

For the discretized problem, the additional integral term can be replaced by a weighted sum of all elements of the vector $\boldsymbol{\rho}$ and the
above considerations can be used to determine the constant $c$ so that the solution of the regularized problem is sufficiently close to the original
one.

%%%%%%%%%%%%%%%%%%%%%%%%%%%%%%%%%%%%%%%%%%%%%%%%%%%%%%%%%%%%%%%%%%%%%%%%%%%%%%%%%%%%%%%%%%%%%%%%%%%%%%%%%%%%%%%%%%%%%%%%%%%%%%%%%%%%%%%%%%%%%%%%%%%%%%
\section*{ACKNOWLEDGEMENTS}
%%%%%%%%%%%%%%%%%%%%%%%%%%%%%%%%%%%%%%%%%%%%%%%%%%%%%%%%%%%%%%%%%%%%%%%%%%%%%%%%%%%%%%%%%%%%%%%%%%%%%%%%%%%%%%%%%%%%%%%%%%%%%%%%%%%%%%%%%%%%%%%%%%%%%%

The authors would like to thank the anonymous reviewer and the editors for their time and effort.

This work was performed within the LOEWE Priority Program Cocoon (www.cocoon.tu-darmstadt.de) supported by the LOEWE research initiative
of the state of Hesse/Germany.

%%%%%%%%%%%%%%%%%%%%%%%%%%%%%%%%%%%%%%%%%%%%%%%%%%%%%%%%%%%%%%%%%%%%%%%%%%%%%%%%%%%%%%%%%%%%%%%%%%%%%%%%%%%%%%%%%%%%%%%%%%%%%%%%%%%%%%%%%%%%%%%%%%%%%%

%%%%%%%%%%%%%%%%%%%%%%%%%%%%%%%%%%%%%%%%%%%%%%%%%%%%%%%%%%%%%%%%%%%%%%%%%%%%%%%%%%%%%%%%%%%%%%%%%%%%%%%%%%%%%%%%%%%%%%%%%%%%%%%%%%%%%%%%%%%%%%%%%%%%%%

\end{document}